\documentclass[11pt,reqno]{amsart}
\usepackage{tikz}
\usepackage{subfig}
\usepackage{epstopdf}
\usepackage{epsfig}
\usepackage{math}
\usepackage{tikz}
\usetikzlibrary{shapes,arrows}
\tikzstyle{decision} = [diamond, draw, fill=blue!20, 
    text width=4.5em, text badly centered, node distance=3cm, inner sep=0pt]
\tikzstyle{block} = [rectangle, draw, fill=blue!20, 
    text width=5em, text centered, rounded corners, minimum height=4em]
\tikzstyle{line} = [draw, -latex']
\tikzstyle{cloud} = [draw, ellipse,fill=red!20, node distance=3cm,
    minimum height=2em]
    
\usetikzlibrary{positioning}
\tikzset{main node/.style={circle,fill=blue!20,draw,minimum size=1cm,inner sep=0pt},  }
\usepackage{hyperref}

\graphicspath{{"../Code/hierarchical models/"},{"../Code/Ricci curvature/three points_g/"},{"../Code/Ricci curvature/Independence model_g/"}}

\begin{document}
\title[Ricci curvature for parametric statistics]{Ricci curvature for parametric statistics\\ via optimal transport} 
\author[Li, Mont\'ufar]{Wuchen Li and Guido Mont\'ufar}
\newcommand{\vr}{\overrightarrow}
\newcommand{\wt}{\widetilde}
\newcommand{\dd}{\mathcal{\dagger}}
\newcommand{\ts}{\mathsf{T}}
\newcommand{\gm}[1]{\textcolor{blue}{#1}}
\newcommand{\gmf}[1]{\textcolor{blue!30!red}{**#1}}
\newcommand{\wc}[1]{\textcolor{red}{#1}}

\maketitle

\begin{abstract}
We elaborate the notion of a Ricci curvature lower bound for parametrized statistical models. 
Following the seminal ideas of Lott-Strum-Villani, we define this notion based on the geodesic convexity  
of the Kullback-Leibler divergence in a Wasserstein statistical manifold, that is, a manifold of probability distributions endowed with a Wasserstein metric tensor structure. Within these definitions, the Ricci curvature is related to both, information geometry and Wasserstein geometry. These definitions allow us to formulate bounds on the convergence rate of Wasserstein gradient flows and information functional inequalities in parameter space. We discuss examples of Ricci curvature lower bounds and convergence rates in exponential family models. 
\end{abstract}

\section{Introduction}

The Ricci curvature lower bound on \emph{sample space} plays a crucial role in various fields, including heat semi-groups \cite{BE} and differential geometry (Brunn-Minkowski inequality) \cite{vil2008}. In particular, it provides sharp bounds for convergence rates of diffusion processes~\cite{BE} and functional inequalities~\cite{OV}. 
In recent years, optimal transport contributes a viewpoint that connects Ricci curvature and information functionals. 
In this study, optimal transport, in particular the $L^2$-Wasserstein metric, introduces a Riemannian structure in probability density space, named density manifold~\cite{Lafferty}. 
The Ricci curvature lower bound in sample space is equivalent to the geodesic convexity of the Kullback-Leibler divergence in the density manifold\footnotemark. 
Following this angle, Lott-Strum-Villani~\cite{Lott_Villani, strum} define the Ricci curvature on non-smooth metric sample spaces, and Erbar-Maas~\cite{Maas2012} introduce it on discrete sample spaces. 

\footnotetext{Geodesic convexity is a synthetic definition. If a function $f$ on manifold $(M, g)$ is second differentiable, then $f$ is $\lambda$-geodesic convex whenever $\operatorname{Hess}_M f\succeq \lambda g$. }

In statistics and machine learning, we often are interested in constructing, or selecting, a density that models the behavior of some observed data, according to some quality criterion. 
Very often we restrict the search to a subset of densities, as this allows us to handle large state spaces and also to incorporate prior knowledge into our search. 
Parametrized statistical models are a ubiquitous and powerful approach. In this paper, we develop the theory of Ricci curvature lower bounds for this situation. 
The Ricci curvature lower bound governs the dissipation rates of the cross entropy. In the context of learning, this corresponds to the rates of convergence of gradient descent methods for minimizing the Kullback-Leibler (KL) divergence and computing information projections. 

The Wasserstein metric tensor of a statistical manifold (a parametrized set of probability densities) has been defined in~\cite{WSM1}. 
A statistical manifold endowed with a Wasserstein metric tensor structure is called Wasserstein statistical manifold. 
We define the Ricci curvature lower bound via geodesic convexity of the KL divergence on a Wasserstein statistical manifold. We obtain a definition of the Ricci-curvature that connects Wasserstein geometry~\cite{vil2008} and information geometry~\cite{Amari, IG}, much in the spirit of~\cite{LiG,WSM1}, and take a natural further step towards connecting the two fields, in particular, relating notions from learning applications and the geometry of the statistical models. 
We focus on discrete sample spaces, which allows us to present a clear picture of the relations deriving from this theory, and leave the details of continuous settings for future work. 

We consider a discrete statistical model described by a tuple $(\Theta, I, \p)$ consisting of a parameter space $\Theta$, a discrete sample space (or state space) $I=\{1,\cdots, n\}$, and a parametrization $\p\colon \Theta\rightarrow \mathcal{P}(I)$. 
Here $\mathcal{P}(I)$ denotes the set of all probability distributions on $I$. We say that $(\Theta, I, \p)$ has Ricci curvature lower bound $\kappa\in \mathbb{R}$ with respect to a given reference measure $q$, if and only if, for any $\theta\in \Theta$, it holds that 
\begin{equation*}
G_F(\theta)+\sum_{a\in I}\Big(d_{\theta\theta}p_a(\theta)\log\frac{p_a(\theta)}{q_a}- \Gamma^{W,a}(\theta)\frac{d}{d{\theta_a}} \operatorname{D}_{\operatorname{KL}}( p(\theta)\|q)\Big)\succeq \kappa G_W(\theta).  
\end{equation*}
Here $G_F$ is the Fisher-Rao metric tensor, $G_W$ is the $L^2$-Wasserstein metric tensor, $d_{\theta\theta}p$ is the second differential of the parameterization, $\Gamma^{W}(\theta)$ are the Christoffel symbols of the Wasserstein statistical manifold, and $\operatorname{D}_{\operatorname{KL}}(p(\theta)\|q)=\sum_{i=1}^np_i(\theta)\log\frac{p_i(\theta)}{q_i}$ is the KL divergence. 
This definition depends on the reference measure $q$. In statistics and learning applications, the reference measure will play the role of a target or empirical data distribution. 
A schematic illustration of the spaces and relations that we consider is provided in Figure~\ref{fig:spaces}. 

The Ricci curvature on discrete state spaces has been studied by many groups. 
(i)~Ollivier~\cite{Ollivier_Ricci} introduces a discrete Ricci curvature via $L_1$-Wasserstein metric. Many inequalities on graphs are shown under this setting; see, e.g.,~\cite{Jost2, Jost, OV2}. 
(ii)~Lin-Yau et al.~\cite{LinYau2, LinYau} also define a Ricci curvature lower bound by heat semi-groups and Bakery-Emery $\Gamma_2$ operators. 
(iii)~Erbar-Maas introduce the Ricci curvature lower bound in \cite{Maas2012} by {by means of equivalence relations with Lott-Strum-Villani} in the Wasserstein probability manifold, under which several information functional inequalities are established. This notion has been studied extensively in \cite{EMR2, EMR3, Erbar, EMR4, EMR1}. 
 However, the notion of a Ricci curvature lower bound on the parameter space of a statistical manifold has not been studied so far. 
Parametrized Wasserstein probability sub-manifolds were not introduced until recently in~\cite{WSM2, WSM1}. 
Our definition of the Ricci curvature lower bound for parametrized statistical models is close in spirit to the definitions by Lott-Strum-Villani and Erbar-Maas. 

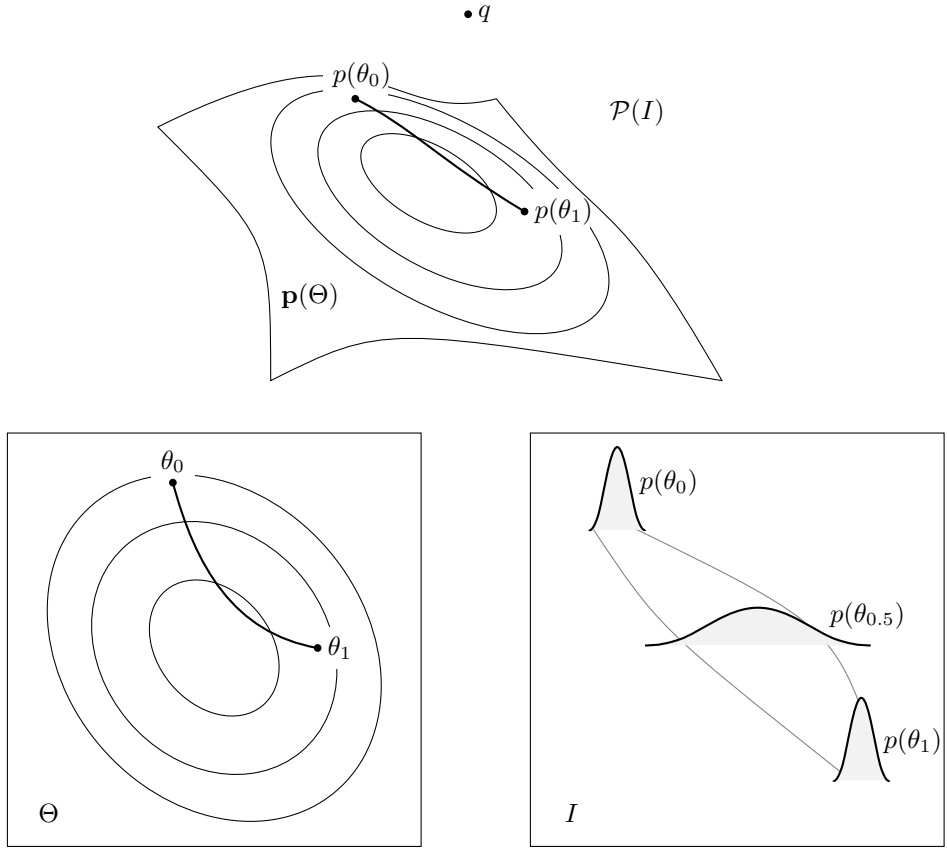
\begin{figure}
	\centering
\newcommand{\simplex}{
	\begin{tikzpicture}[x=1.5cm,y=1.5cm]
	\draw (0.5,.25) circle [x radius=2.5cm, y radius=1.2cm, rotate=-30];
	\draw (0.5,.35) circle [x radius=1.8cm, y radius=.9cm, rotate=-30];
	\draw (0.4,.5) circle [x radius=1cm, y radius=.5cm, rotate=-30];
	
	\draw[fill=black,color=black] (1.25,.25) circle (1pt)node[right, fill = white]{\small$p(\theta_1)\!\!\!$};
	
	\draw[black] (-1,-1.25) 
	.. controls (-1,0) and (-1,0) .. (-2,1)
	.. controls (0,2) and (0,1) .. (1,1.25) 
	.. controls (2,0) and (2,0.5) .. (3,-1.25) 
	.. controls (0,-.75) and (0,-.75) .. (-1,-1.25); 
	
	\draw[thick] (-.25,1.25) .. controls (.25,1) and (.5,.7) .. (1.25,.25); 
	
	\draw[fill=black,color=black] (-.25,1.25) circle (1pt)node[above, fill = white]{\small$p(\theta_0)\!\!\!$};
	\draw[fill=black,color=black] (-.25,1.25) circle (1.25pt);
	\draw[fill=black,color=black] (1.25,.25) circle (1.25pt);
	
	\draw[fill=black,color=black] (.75,2) circle (1pt)node[right, fill = white]{\small$q\!\!\!$};
	\draw[fill=black,color=black] (.75,2) circle (1.25pt);
	
	\node at (-.65,-.5) {\small$\mathbf{p}(\Theta)$};
	\node at (2.25,1.125) {\small$\mathcal{P}(I)$};
	\end{tikzpicture}
}

\newcommand{\parameter}{ 
	\begin{tikzpicture}[x=1.1cm,y=1.1cm]
	\newcommand{\para}{\draw[fill = gray!10,thick] plot[smooth, tension=.7] coordinates {(-2,1.5)(-1.85,1.65) (-1.5,2.5) (-1.15,1.65)(-1,1.5)};}
	\draw[black] (-2.5,-2.4) -- (-2.5,2.6) -- (2.5,2.6) -- (2.5,-2.4) -- (-2.5,-2.4); 
	
	\draw (0,0) circle [x radius=2.5cm, y radius=2cm, rotate=-50];
	\draw (0,0) circle [x radius=1.8cm, y radius=1.5cm, rotate=-50];
	\draw (0,0) circle [x radius=1cm, y radius=.75cm, rotate=-50]; 
	\draw[thick] (-.5,2) to[out=-75,in=170] (1.25,0);
	
	\draw[fill=black,color=black] (-.5,2) circle (1pt) node[above,fill = white]{\small$\theta_0$};
	\draw[fill=black,color=black] (1.25,0) circle (1pt)node[right, fill = white]{\small$\theta_1$};
	\draw[fill=black,color=black] (-.5,2) circle (1.25pt);
	\draw[fill=black,color=black] (1.25,0) circle (1.25pt);
	
	\node at (-2,-2) {\small$\Theta$};
	\end{tikzpicture}
}

\newcommand{\states}{ 
	\begin{tikzpicture}[x=1.1cm,y=1.1cm]
	\newcommand{\para}{\draw[fill = gray!10,thick] plot[smooth, tension=.7] coordinates {(-2,1.5)(-1.85,1.65) (-1.5,2.5) (-1.15,1.65)(-1,1.5)};}
	\draw[] (-2.5,-2.4) -- (-2.5,2.6) -- (2.5,2.6) -- (2.5,-2.4) -- (-2.5,-2.4); 
	
	\draw[gray] (-1.75,1.45) .. controls (-.95,.25) .. (1.25,-1.5); 
	\draw[gray] (-1.25,1.45) .. controls (1.2,.25) .. (1.75,-1.5); 
	
	\node[above] at (-1.45,1.25) {\begin{tikzpicture}[x=.75cm, y=1.1cm]
		\para
		\end{tikzpicture}};
	\node[right] at (-1.3,2) {\small$p(\theta_0)$};
	
	\node[] at (.25,.25) {\begin{tikzpicture}[x=3cm, y=.5cm]
		\para
		\end{tikzpicture}};
	\node[right] at (1,0.4) {\small$p(\theta_{0.5})$};
	
	\node[] at (1.5,-1.125) {\begin{tikzpicture}[x=.75cm, y=1.1cm]
		\para
		\end{tikzpicture}};
	\node[right] at (1.65,-1.125) {\small$p(\theta_1)$};
	
	\node at (-2,-2) {\small$I$};
\end{tikzpicture}
}

\begin{tikzpicture}
\node at (0,6){\simplex}; 
\node at (-3,0){\parameter}; 
\node at (4,0){\states};
\end{tikzpicture}
	\caption{
		Our discussion involves a state space $I$, a parameter space $\Theta$, and a parametrized set $\p(\Theta)$ in the space $\mathcal{P}(I)$ of probability distributions on $I$. 
		For a reference measure $q\in\mathcal{P}(I)$, a positive Ricci curvature lower bound implies that the Wasserstein geodesic connecting two distributions, $p(\theta_0)$ and $p(\theta_1)$, `bends' towards $q$. 
		The figure depicts the  geodesic as a thick curve, together with the level sets of $\operatorname{D}_{\operatorname{KL}}(p(\cdot)\|q)$, in $\Theta$ and $\p(\Theta)$. 
		In terms of the state space $I$, when $q$ is uniform, 
		a decrease of the KL divergence with respect to $q$ corresponds to an increase of the entropy, 
		meaning that along the geodesic, the `volume' of states under the distributions `bulges'. This corresponds to the synthetic notion of positive curvature in sample space. 
		Note how the geodesics are constrained to lie within the model $\p(\Theta)$, which in general does not contain $q$. 
		See Definition~\ref{def}, Theorem~\ref{RIWthm}, Proposition~\ref{thm}, and Figures~\ref{fig:Pythagorean} and~\ref{fig:curv} for more details.}
	\label{fig:spaces}
\end{figure}

This paper is organized as follows. 
In Section~\ref{sec2}, we briefly review the connections between Ricci curvature, optimal transport, and KL divergence. We further demonstrate these connections in the context of information projections. 
In Section~\ref{sec:WSM}, we introduce Wasserstein statistical manifolds in parameter space. This is intended as a short review of the definitions from~\cite{WSM1}.  
We derive the Fokker-Planck equation on parameter space, which is the Wasserstein gradient flow of the KL divergence. 
The main technical contributions of this paper are contained in Section~\ref{sec4}. 
We describe the convergence rate of the Fokker-Planck equation in terms of a Ricci curvature lower bound. 
Further, we use the notion of Ricci curvature lower bound to establish information functional inequalities. 
We also discuss methods to estimate the Ricci curvature lower bound in practice. 
In Section~\ref{sec5}, we present experiments on small examples of exponential families. These allow us to illustrate the notions introduced in the paper, and gain more intuition about their meaning.

\section{Ricci curvature and information projections} 
\label{sec2}

In this section, we review the connection of optimal transport and information theory put forward in Villani's book \cite{vil2008}, and we further connect with the notion of information projections described by Csisz\'ar-Shields~\cite{Csiszar:2004:ITS:1166379.1166380}. 
In later sections we will develop these connections for the case of parametric statistical models.  

\subsection{Wasserstein geometry} 
Consider a continuous measure space $(\Omega, g^\Omega, q)$. 
Here $\Omega$ is a finite dimensional compact smooth Riemannian manifold without boundary, $g^\Omega$ is its metric tensor, $dx$ is the volume from of $\Omega$, and $q\in C^{\infty}(\Omega)$ is the measure volume form with $\int_\Omega q(x)=1$, $q(x)>0$. 
The Ricci curvature tensor on $(\Omega, g^{\Omega}, q)$ refers to 
\begin{equation}\label{eq:rclb}
\operatorname{Ric}=\operatorname{Ric}_\Omega-\operatorname{Hess}_\Omega \log q,
\end{equation}
where $\operatorname{Ric}_\Omega$ denotes the Ricci curvature on $\Omega$ and $\operatorname{Hess}_\Omega$ is the Hessian operator on $\Omega$. 
{Note that this notion of curvature depends on the reference measure $q$. 
Later in our discussion, the reference measure will play the role of a target or empirical data distribution.} 
 
On the one hand, optimal transport, in particular the $L^2$-Wasserstein metric, introduces an infinite-dimensional Riemannian structure in density space. 
In the context of our discussion, consider the set of smooth and strictly positive densities 
\begin{equation*}
\mathcal{P}_+(\Omega)=\Big\{\rho \in C^{\infty}(\Omega)\colon \rho(x)>0,~\int_\Omega\rho(x)dx=1\Big\}. 
\end{equation*}
The tangent space of $\mathcal{P}_+(\Omega)$ at $\rho\in \mathcal{P}_+(\Omega)$ is given by 
 \begin{equation*}
 T_\rho\mathcal{P}_+(\Omega) = \Big\{\sigma\in C^{\infty}(\Omega)\colon \int_\Omega\sigma(x) dx=0 \Big\}.
 \end{equation*}  
 
 \begin{definition}[$L^2$-Wasserstein metric tensor]\label{eqn:riemannian whole density space}
 Define the inner product $g_\rho\colon {T_\rho}\mathcal{P}_+(\Omega)\times {T_\rho}\mathcal{P}_+(\Omega)\rightarrow \mathbb{R}$ by
 \begin{equation*}
 g_\rho(\sigma_1, \sigma_2)=\int_\Omega \sigma_1(x)(-\Delta_\rho)^{\dd}\sigma_2(x) dx,
  \end{equation*}
 where $\Delta^{\dd}_\rho\colon {T_\rho}\mathcal{P}_+(\Omega)\rightarrow{T_\rho}\mathcal{P}_+(\Omega)$ is the inverse of elliptical operator $\Delta_\rho = \nabla\cdot(\rho \nabla)$. 
 Here $\nabla$ and $\nabla\cdot$ are the gradient and divergence operators in $\Omega$, respectively. 
 \end{definition}
 
 Following \cite{Lafferty}, we call $(\mathcal{P}_+(\Omega), g)$ a Wasserstein density manifold or a Wasserstein manifold for short. 
 The metric tensor introduces a variational formulation of a metric function. 
More precisely, the square of the $L^2$-Wasserstein metric function is equal to the geometric energy (action) of geodesics in the Wasserstein manifold. 
For any $\rho_0,\rho_1\in \mathcal{P}_+(\Omega)$, the $L^2$-Wasserstein metric function is defined as 
 \begin{equation*}
 W(\rho_0, \rho_1)^2=\inf_{}~\Big\{\int_0^1 {g_{\rho_t}(\partial_t\rho_t, \partial_t\rho_t) }dt \colon \rho_t  \in\mathcal{P}_+(\Omega), t\in[0,1] \Big\}.
 \end{equation*}
One can extend the definitions from $\mathcal{P}_+(\Omega)$ to the set $\mathcal{P}_2(\Omega)$ of Borel probability measures with finite second moments. 
It is well known that the $L^2$-Wasserstein metric defines a metric function on $\mathcal{P}_2(\Omega)$, and {hence} $(\mathcal{P}_2(\Omega), W)$ forms a length space. 
See related analytical treatments in~\cite{vil2008}. 
 
\subsection{Wasserstein gradient flow of the KL divergence}
On the other hand, 
information theory considers a particular functional on density space, namely the KL divergence. 
Given a smooth reference measure $q\in \mathcal{P}_+(\Omega)$, the KL divergence of a given $\rho$ with respect to $q$ is defined by 
\begin{equation*}
\operatorname{D}_{\operatorname{KL}}(\rho \|q )=\int_{\Omega}\rho(x)\log\frac{\rho(x)}{q(x)}dx.
\end{equation*}
Notice that the KL divergence is precisely the free energy.  Indeed, if we write $q(x)=\frac{1}{K}e^{-V(x)}$ with $K=\int_{\Omega}e^{-V(x)}dx$, we see that 
\begin{equation*}
\begin{split}
\operatorname{D}_{\operatorname{KL}}(\rho \|q )=&\int_{\Omega}\rho(x)\log\rho(x)dx+\int_{\Omega}V(x)\rho(x)dx+\log K\\
=&- H(\rho)+\mathbb{E}_\rho[V(X)] +\log K\ ,
\end{split}
\end{equation*}
where $H(\rho)=-\int_\Omega\rho(x)\log\rho(x)\;dx$ is the Boltzmann-Shannon entropy, $X$ is a random variable satisfying the law of density $\rho$ and $\mathbb{E}$ is the expectation operator.

The Ricci curvature on sample space is related both to the KL divergence and the $L^2$-Wasserstein metric tensor. 
This interaction starts with the gradient flow of the KL divergence in the Wasserstein manifold $(\mathcal{P}_+(\Omega), g)$, 
which describes the time evolution of the density following the negative Wasserstein gradient of the KL divergence: 
\begin{equation}\label{FPE}
\begin{split}
\frac{\partial \rho_t}{\partial t}=&-\operatorname{grad}_W\operatorname{D}_{\operatorname{KL}}(\rho\|q) \\
=&\nabla\cdot(\rho_t\nabla (\log\frac{\rho_t}{q}+1))\\
=&\nabla\cdot(\rho_t\nabla V)+\Delta \rho_t. 
\end{split}
\end{equation}
The second line is by Definition~\ref{eqn:riemannian whole density space} of the Wasserstein metric tensor. 
The last equality holds since $q(x)=\frac{1}{K}e^{-V(x)}$ and $\nabla\cdot(\rho \nabla \log\rho)=\nabla\cdot(\nabla\rho)=\Delta\rho$. 

It is worth noting that there are several perspectives based on~\eqref{FPE}. Firstly, the flow~\eqref{FPE} is a well-known dynamics called \emph{Fokker-Planck equation} (FPE). 
It describes the probability transition equation of drift diffusion process 
\begin{equation*}
\dot X_t=-\nabla V(X_t)+\sqrt{2}\dot B_t,
\end{equation*}
where $B_t$ is the canonical Brownian motion in sample space. Secondly, along the flow~\eqref{FPE}, the KL divergence converges to zero. I.e.~$\rho_t$ converges to the minimizer of the KL divergence (free energy), known as the Gibbs measure, $q(x)=\frac{1}{K}e^{-V(x)}$. This reminds of iterative methods for computing information projections~\cite{Csiszar:2004:ITS:1166379.1166380} in statistics and machine learning. In this context, one seeks to reproduce the behavior of a teacher system in terms of a model. To this end, the learning rule proceeds by adjusting the model parameters so as to maximize the likelihood of the observations, which is equivalent to minimizing the divergence, for instance using Wasserstein gradient descent. The flow is the continuous limit of the gradient descent learning rule. 
We shall go to this connection shortly, in Section~\ref{sec:IP}. 

\subsection{Dissipation rates and the Ricci curvature lower bound}

As it turns out, the Ricci curvature lower bound governs the exponential dissipation rate of \eqref{FPE} towards the Gibbs measure $q$. 
In the setting of learning, this corresponds precisely to the exponential rate of convergence of the learning dynamics. 
To see this, the following calculations in dynamical system are used. One can find the convergence rate of \eqref{FPE} by comparing the ratio between the first and second time derivatives along the flow. 
By some computations, the first time derivative of the KL divergence along the flow is found to be equal to 
\begin{equation*}
\begin{split}
-\frac{d}{dt}\operatorname{D}_{\operatorname{KL}}(\rho_t\|q)=&g_{\rho_t}(\partial_t\rho_t, \partial_t\rho_t)\\
=&\int_{\Omega}\Gamma(\log\frac{\rho_t}{q}, \log\frac{\rho_t}{q})\rho_t\;dx,
\end{split}
\end{equation*}
while the second time derivative is given by 
\begin{equation*}
\begin{split}
\frac{d^2}{dt^2}\operatorname{D}_{\operatorname{KL}}(\rho_t\|q)=&2\operatorname{Hess}_W\operatorname{D}_{\operatorname{KL}}(\rho_t\|q)(\partial_t\rho_t, \partial_t\rho_t)\\
=&2\int_{\Omega}\Gamma_2( \log\frac{\rho_t}{q}, \log\frac{\rho_t}{q})\rho_t \;dx\ .
\end{split}
\end{equation*}
Here $\operatorname{Hess}_W$ is the Hessian operator with respect to the Wasserstein metric tensor, and $\Gamma$ and $\Gamma_2$ are the Bakery-Emery operators defined by
$$
\Gamma(f,f)=g^\Omega(\nabla f, \nabla f),
$$
and 
$$
\Gamma_2(f,f)=(\operatorname{Ric}_\Omega-\operatorname{Hess}_\Omega \log q)(\nabla f,\nabla f)+\operatorname{tr}(\operatorname{Hess}_\Omega f, \operatorname{Hess}_\Omega f),
$$
where $\textrm{Ric}_\Omega$ is the Ricci curvature tensor on $\Omega$, $\textrm{Hess}_\Omega$ is the Hessian operator on $\Omega$, and $\textrm{tr}$ is the trace operator.
By the above formulas, the ratio between $\frac{d}{dt}\operatorname{D}_{\operatorname{KL}}(\rho_t\|q)$ and $\frac{d^2}{dt^2}\operatorname{D}_{\operatorname{KL}}(\rho_t\|q)$ relates to the integral version of $\Gamma$, $\Gamma_2$, i.e. the expectation values of the operators $\Gamma$, $\Gamma_2$. 
 Notice that $\operatorname{tr}(\operatorname{Hess}_\Omega f, \operatorname{Hess}_\Omega f)\geq 0$. 
 Classical results \cite{vil2008} show that the lower bound of Ricci curvature governs the smallest ratio between $\frac{d}{dt}\operatorname{D}_{\operatorname{KL}}(\rho_t\|q)$ and $\frac{d^2}{dt^2}\operatorname{D}_{\operatorname{KL}}(\rho_t\|q)$, which further gives the exponential convergence rate of \eqref{FPE}. 
In addition, the above computation demonstrate that the lower bound of the Ricci curvature, informally speaking, is equivalent to the smallest eigenvalue of the Hessian operator of the KL divergence in the Wasserstein manifold.
 \begin{theorem}\label{thmR}
Given $\kappa\in\mathbb{R}$ and $q(x)\in \mathcal{P}_+(\Omega)$, the following statements are equivalent. \begin{itemize}
\item[(i)] $\kappa$ is a Ricci curvature lower bound {of} $(\Omega, g^\Omega, q)$. I.e. 
{$\kappa$ is the largest number for which, uniformly over $\Omega$, } $$\operatorname{Ric}=\operatorname{Ric}_\Omega-\operatorname{Hess}_\Omega \log q\succeq \kappa g^\Omega;$$ 
\item[(ii)] $\Gamma_2(f,f)\geq \kappa\Gamma(f,f)$, for any $f\in C^{\infty}(\Omega)$;
\item[(iii)] For any constant speed geodesic $\rho_t$, $t\in[0,1]$, connecting $\rho_0$ and $\rho_1$ in $(\mathcal{P}_2(\Omega), W)$, 
\begin{equation*}
\operatorname{D}_{\operatorname{KL}}(\rho_t\|q)\leq (1-t) \operatorname{D}_{\operatorname{KL}}(\rho_0\|q)+t \operatorname{D}_{\operatorname{KL}}(\rho_1\|q)-\frac{\kappa}{2}t(1-t)W(\rho_0, \rho_1)^2.
\end{equation*}
\end{itemize}
\end{theorem}

Theorem \ref{thmR} opens the door to define a notion of Ricci curvature lower bound on sample space via its equivalent statements. 
In the literature, Bakery-Emery \cite{BE} define the Ricci curvature lower bound by applying (ii) for smooth Riemannian sample spaces, 
while Lott-Strum-Villani~\cite{Lott_Villani, strum} define it using (iii) for non-smooth metric sample spaces, and Erbar-Maas~\cite{Maas2012} define it by (iii) in a discrete sample space. 
In this paper, we shall define the notion of Ricci curvature lower bound for parametric statistics taking an approach based on (iii), {known as the geodesic convexity property of the KL divergence}. 

\subsection{Learning in a parametrized model} 
In statistics and machine learning applications one often considers a parametrized set $\{p(\cdot;\theta)\colon \theta\in\mathbb{R}^d\}$ of candidate probability distributions from which one wishes to choose one to model the distribution of some given data. 

One motivation for using parametrized models is reducing the dimensionality associated with large state spaces. For instance, we may be considering a state space consisting of images presented as arrays of pixel intensities, corresponding to $\{0,1\}^n$ with $n$ easily in the order of thousands. 
In this case, storing a probability distribution as a vector $p\in\mathbb{R}^{2^n}$ of individual probabilities $p(x)$, $x\in\{0,1\}^n$, is an impossibility. With a parametric model, instead of storing the probability vector, we store only a parameter vector $\theta\in\mathbb{R}^d$, with a more manageable $d$, and fix a mapping that allows us to recover individual values $p(x;\theta)$ of the probability distribution for a given $x$, or, in other cases, which allows us to generate samples from $p(\cdot;\theta)$ that we can also use to estimate any expectation values of interest. Reducing the dimensionality is useful not only in terms of storage, but also in a statistical sense, in relation to overfitting.  Without going into details, the richer the class of hypotheses, with more free parameters, the more prone we are to fitting statistical nuisances of the data, instead of capturing the true general behavior of the data. By working with a parametrized model, we can incorporate priors into the learning system and limit its vulnerability to overfitting.  

When working with a parametrized model, obtaining the best possible hypothesis, e.g. the maximizer of the data likelihood, is usually a non trivial problem and one has to resort to iterative methods. 
A relevant question then is the computational effort needed for this. 
In particular, one is interested in the number of iterations needed until reaching a solution that is within $\epsilon$ of the best possible. 
The Ricci curvature can be regarded as a way to obtain bounds on the convergence rate of gradient optimization of the KL divergence for a given a target distribution, uniformly over the start distribution. 
We elaborate on this in the next paragraph. The situation is illustrated in Figures~\ref{fig:spaces},~\ref{fig:Pythagorean}, and~\ref{fig:curv}. 

\subsection{I-projections}\label{sec:IP}
In the context of information theory and statistics, Csisz\'ar-Shields~\cite{Csiszar:2004:ITS:1166379.1166380} define the \emph{I-projection} of a distribution $Q$ onto a non-empty closed convex set $\mathcal{N}$ of distributions as the point $P^\ast\in\mathcal{N}$ such that 
$$
D(P^\ast\| Q) = \min_{P\in\mathcal{N}} D(P\|Q). 
$$
The notion of I-projection considers the minimization of the KL divergence with respect to the first argument, but it is also relevant in the context of maximum likelihood estimation, where the minimization is with respect to the second argument. 
Given an empirical data distribution $P$, a maximum likelihood estimator over a set $\mathcal{E}$ is a point $P^\ast\in \overline{\mathcal{E}}$ (the closure of $\mathcal{E}$), with 
$$
D(P\|P^\ast) = \inf_{Q\in\mathcal{E}}D(P\|Q). 
$$ 
If we consider an exponential family model $\mathcal{E}=\{p \propto Q\exp(\theta^{\ts} F) \colon \theta\in \mathbb{R}^d \}$ on a finite state space $I$, 
with sufficient statistics $F\colon I\to \mathbb{R}^d$ and reference measure $Q\in\mathcal{P}_+(I)$, then the maximum likelihood estimator $P^\ast$ of the target distribution $P$ can be obtained as the I-projection of $Q$ onto the orthogonal linear family defined by  $\mathcal{N}=\{ p \colon \sum_x F(x)p(x)= \sum_x F(x)P(x)\}$. 
We have namely that
$$
P^\ast=\operatorname{argmin}_{Q\in\mathcal{E}}D(P\|Q) = \operatorname{argmin}_{P\in\mathcal{N}}D(P\|Q). 
$$
This is a consequence of the well known Pythagorean relation~\cite{Csiszar:2004:ITS:1166379.1166380,Amari} illustrated in Figure~\ref{fig:Pythagorean}. 
\begin{figure}
\centering 
\begin{tikzpicture}[x=1.5cm,y=1.5cm]
\draw[gray, thick] (-1,-1) .. controls (-1,0) and (-1,0) .. (-2,1);
\draw[gray, thick] (2.6,-1) .. controls (2,0.5) and (2,0) .. (1,1);
\draw[gray, thick] (-2,1) .. controls (0,2) and (0,1) .. (1,1);
\draw[above] (-1.8,1.2) node {\textcolor{black}{$\mathcal{E}$}}; 	

\draw[gray, thick,-] (0,-1.125) -- (0,2.5); 
\draw[black, thick,-] (0,.5) -- (0,2); 
\draw[left] (0,2.5) node {\textcolor{black}{$\mathcal{N}$}}; 

\draw (0.05,1.15)node [left]{\small$D(P\|P^\ast)$};	

\draw (0.85,-.05) node[left,fill=white] {\small$D(P^\ast\|Q)$};	
\draw[black, thick] (0,.5) .. controls (.5,.35) .. (1.25,-.125); 

\draw[black, thick] (0,2) .. controls (.5,1.25) .. (1.25,-.125) node[midway, right] {\small$D(P\|Q)$};	

\draw [black,fill, fill=black] (1.25,-.125) circle [radius=1.2pt];	 
\draw[right] (1.25,-.125) node {\textcolor{black}{$Q$}}; 	

\draw [black,fill, fill=black] (0,0.5) circle [radius=1.2pt]; 
\draw[left] (0,0.5) node {\textcolor{black}{$P^\ast$}}; 

\draw [black,fill, fill=black] (0,2) circle [radius=1.2pt]; 
\draw[right] (0,2) node {\textcolor{black}{$P$}}; 
\draw[gray, thick] (-1,-1) .. controls (0,-.5) and (0,-.5) .. (2.6,-1);
\end{tikzpicture}
	\caption{For a distribution $Q$ in an exponential family $\mathcal{E}$ and a distribution $P$ in an orthogonal linear family $\mathcal{N}$, the Pythagorean relation holds: $D(P\|Q) = D(P\| P^\ast) + D(P^\ast\|Q)$, where $P^\ast$ is the unique intersection point of $\mathcal{E}$ and $\mathcal{N}$. }
	\label{fig:Pythagorean}
\end{figure}
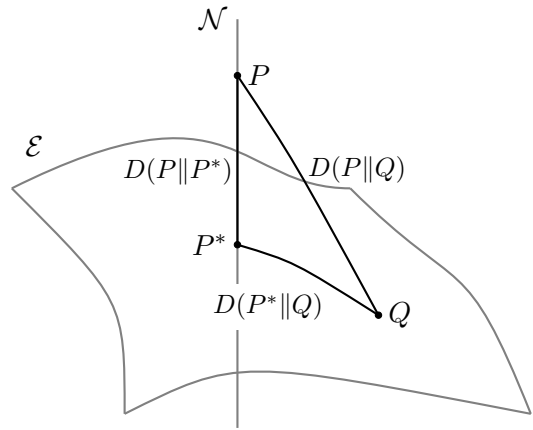

Csisz\'ar and Shields~\cite{Csiszar:2004:ITS:1166379.1166380} consider iterative methods for computing I-projections, and obtain upper bounds on the divergence along the resulting parameter trajectories, 
which describe the convergence to the optimum value. 
For two sets of distributions, $\mathcal{P}$ and $\mathcal{Q}$, together with two functions $D(\cdot,\cdot)\colon \mathcal{P}\times\mathcal{Q}\to \mathbb{R}$ and $\delta(\cdot,\cdot)\colon \mathcal{P}\times \mathcal{P}\to \mathbb{R}$, satisfying certain conditions, they describe an iterative algorithm (alternating divergence minimization) which iterates $p_n\in \mathcal{P}$ and $q_n\in\mathcal{Q}$, and give an upper bound of the form 
\begin{equation}\label{I_P}
D(p_{n+1}, q_n)  - D_{\text{min}}\leq \delta(p_\infty, p_n) - \delta(p_\infty, p_{n+1}) . 
\end{equation}

In this paper we are in the special setting where $\mathcal{Q}=\{q\}$ and $\mathcal{P}$ is the set of all densities. There is a natural connection between~\eqref{I_P} and the Fokker-Planck-equation~\eqref{FPE}. 
Indeed, setting $D$ as the KL divergence and 
$q_n=q=p_\infty$, $p_n=\rho_t$, $p_{n+1}=\rho_{t+\Delta t}$, where $\Delta t$ is the step size, we demonstrate in Proposition~\ref{thm} that we can substitute 
\begin{equation*}
\delta(q, p)=\frac{1}{2\kappa\Delta t}\operatorname{D}_{\operatorname{KL}}(p\|q), 
\end{equation*}
where $\kappa$ is the Ricci curvature lower bound that we will define later on. 
Strictly speaking, for this correspondence, we need to assume that $\kappa>0$, which is a natural requirement similar to requiring that the KL divergence is geodesic convex in set $\mathcal{P}$. Each step dissipation in \eqref{I_P} then gives 
\begin{equation*}
\begin{split}
\operatorname{D}_{\operatorname{KL}}(\rho_{t+\Delta t}\| q)  - \operatorname{D}_{\text{min}}
\leq & \delta(p_\infty, p_n) - \delta(p_\infty, p_{n+1})\\
= & \frac{1}{2\kappa\Delta t}\Big\{\operatorname{D}_{\operatorname{KL}}(\rho_t\|q) - \operatorname{D}_{\operatorname{KL}}(\rho_{t+\Delta t}\| q)\Big\}\\
=&-\frac{1}{2\kappa}  \frac{d}{dt}\operatorname{D}_{\operatorname{KL}}(\rho_t\|q)+o(\Delta t).
\end{split}
\end{equation*}
In other words, the Fokker-Planck equation is a monotone information projection flow, in which the dissipation quantity is governed by the difference of relative entropy divided by twice the Ricci curvature lower bound. 
In the limit where $\Delta t$ goes to zero, 
\begin{equation*}
\operatorname{D}_{\operatorname{KL}}(\rho_{t}\| q)  - \operatorname{D}_{\text{min}}\leq -\frac{1}{2\kappa}\frac{d}{dt}\operatorname{D}_{\operatorname{KL}}(\rho_t\|q).
\end{equation*}
Gr\"onwall's inequality then implies that this I-projection flow \eqref{FPE} converges to the minimizer at the rate of $e^{-2\kappa t}$, i.e. 
\begin{equation*}
\operatorname{D}_{\operatorname{KL}}(\rho_{t}\| q)  - \operatorname{D}_{\text{min}}\leq e^{-2\kappa t}\Big( \operatorname{D}_{\operatorname{KL}}(\rho_{0}\| q)  - \operatorname{D}_{\text{min}} \Big) . 
\end{equation*}
The above shows that the learning rate for the Fokker-Planck equation is linear, whose lower bound is governed by $\kappa$. 
Following these connections, we {will} {pursue the definition} of the Ricci curvature lower bound on parameter space. 
The convergence rate, in relation to the Ricci curvature lower bound and the geodesic convexity of the KL divergence, is illustrated schematically in Figure~\ref{fig:curv}. 
More details will be provided in Proposition~\ref{thm}. 

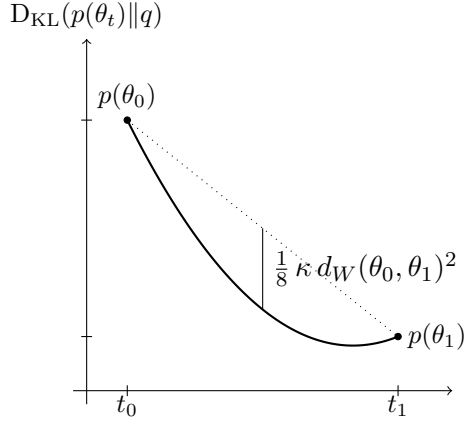
\begin{figure}
	\centering
	\newcommand{\DKL}{\operatorname{D}_{\operatorname{KL}}}
	\definecolor{bl}{rgb}{0,0,0}
	\begin{tikzpicture}[x=1.2cm,y=1.2cm]
	\draw[->, scale=3] (-.2,-.1) -- (1.2,-.1) node[right]{}; 
	\draw[->, scale=3] (-.15,-.15) -- (-.15,1.2)node[above]{\small $\DKL(p(\theta_t)\|q)$}; 
	
	\draw[dotted,scale=3] (0,.9) -- (1,0.1) ;
	\draw[samples=100,scale=3,domain=0:1,smooth,variable=\x,bl,thick] plot ({\x}, {.1+.8*(1-\x) - 1.2*\x*(1-\x)});
	
	\draw[-, scale=3] (-.17,.9) -- (-.13,.9); 
		
	\draw[fill=bl,color=bl,scale=3] (0,.9) circle (0.4pt) node[above]{\small$p(\theta_0)$};
	
	\draw[-, scale=3] (-.17,0.1) -- (-.13,0.1); 
	
	\draw[fill=bl,color=bl,scale=3] (1,0.1) circle (0.4pt)node[right]{\small$p(\theta_1)$};

	\draw[-, scale=3] (0,-.12) -- (0,-.08) node[below]{\small$t_0$};
	\draw[-, scale=3] (1,-.12) -- (1,-.08)node[below]{\small$t_1$};
	
	\draw[-, scale=3] (.5,.5) -- (.5,.2)node[midway,right]{$\tfrac18 \,\kappa\, d_W(\theta_0,\theta_1)^2$};
	\end{tikzpicture}
	\caption{Illustration of the Ricci curvature lower bound $\kappa$ in connection to the geodesic convexity of the KL divergence and the rate of convergence of the information projection flow. 
		Here $p(\theta_t)$ is a Wasserstein geodesic connecting $p(\theta_0)$ and $p(\theta_1)$. 
		{When $q=p(\theta_1)$, the KL divergence $\operatorname{D}_{\operatorname{KL}}(p(\theta_t)\|q)$ is monotonically decreasing}. } 
	\label{fig:curv}
\end{figure}

\section{Wasserstein statistical manifolds}\label{sec:WSM}

In preparation for the definitions and results on the Ricci curvature that we will present in the next section, we briefly review the definition of a Wasserstein statistical manifold with discrete sample space from~\cite{WSM1}, and present the Fokker-Planck equation on parameter space.

\subsection{Wasserstein geometry on the probability simplex}
We recall the definition of discrete probability simplex with $L^2$-Wasserstein Riemannian metric. Consider the discrete sample space $I=\{1,\cdots, n\}$. 
The probability simplex on $I$ is the set 
\begin{equation*}
\mathcal{P}(I) = \Big\{(p_1,\cdots, p_n)\in \mathbb{R}^n \colon \sum_{i=i}^n p_i=1,\quad  p_i\geq 0\Big\}. 
\end{equation*}
Here $p=(p_1,\ldots, p_n)$ is a probability vector with coordinates $p_i$ corresponding to the probabilities assigned to each node $i\in I$. 
The probability simplex $\mathcal{P}(I)$ is a manifold with boundary. 
We denote the interior by $\mathcal{P}_+(I)$. 
This consists of the strictly positive probability distributions, with $p_i>0$ for all $i\in I$. 
To simplify the discussion, we will focus on the interior $\mathcal{P}_+(I)$. 
For the studies related to the boundary $\partial\mathcal{P}(I)$, we refer the reader to~\cite{LiG}. 

Next we define the $L^2$-Wasserstein metric tensor on $\mathcal{P}_+(I)$, which also encodes the metric tensor of discrete states $I$. 
We need to give a ground metric notion on sample space. We do this in terms of a undirected graph with weighted edges, $G=(I, E, \omega)$, where $I$ is the vertex set, $E\subseteq {I\choose 2}$ is the edge set, and $\omega=(\omega_{ij})_{i,j\in I}\in \mathbb{R}^{n\times n}$ is a matrix of edge weights satisfying 
$$
\omega_{ij}=
\begin{cases}
\omega_{ji}>0, & \text{if $(i,j)\in E$}\\
0, & \text{otherwise}
\end{cases}.
$$ 
The set of neighbors (adjacent vertices) of $i$ is denoted by $N(i)=\{j\in V\colon (i,j)\in E\}$. 
The normalized volume form on node $i\in I$ is given by $d_i=\frac{\sum_{j\in N(i)}\omega_{ij}}{\sum_{i=1}^n\sum_{i'\in N(i)}\omega_{ii'}}$. 

The graph structure $G=(I, E, \omega)$ induces a graph Laplacian matrix function.  \begin{definition}[Weighted Laplacian matrix]
	\label{def2}
	Given an undirected weighted graph $G=(I,E,\omega)$, with $I=\{1,\ldots, n\}$, the matrix function $L(\cdot):\mathbb{R}^n\rightarrow \mathbb{R}^{n\times n}$ is defined by
	\begin{equation*}
	L(p)=D^{\ts}\Lambda(p)D,\quad p=(p_i)_{i=1}^n\in \mathbb{R}^n,
	\end{equation*}
	where 
	\begin{itemize}
		\item 
		$D \in \mathbb{R}^{|E|\times n}$ is the discrete gradient operator defined by 
		\begin{equation*} 
		D_{(i,j)\in {E}, k\in V}=\begin{cases}
		\sqrt{\omega_{ij}}, & \text{if $i=k$, $i>j$}\\ 
		-\sqrt{\omega_{ij}}, & \text{if $j=k$, $i>j$}\\
		0, & \text{otherwise}
		\end{cases},
		\end{equation*}
		\item $-D^{\ts}\in \mathbb{R}^{n\times |E|}$ is the oriented incidence  matrix,
		and 
		\item $\Lambda(p)\in \mathbb{R}^{|E|\times |E|}$ is a weight matrix depending on $p$, 
		\begin{equation*}
		\Lambda(p)_{(i,j)\in E, (k,l)\in E}=\begin{cases}
		\frac{1}{2}(\frac{1}{d_i}p_i+\frac{1}{d_j}p_j) & \text{if $(i,j)=(k,l)\in E$}\\ 
		0 & \text{otherwise}
		\end{cases}.
		\end{equation*}
	\end{itemize}
\end{definition}
The Laplacian matrix function $L(p)$ is the discrete analog of the weighted Laplacian operator $-\nabla\cdot(\rho \nabla)$ from Definition~\ref{eqn:riemannian whole density space}. 

We are now ready to present the $L^2$-Wasserstein metric tensor. Consider the tangent space of $\mathcal{P}_+(I)$ at $p$, 
\begin{equation*}
T_p\mathcal{P}_+(I) = \Big\{(\sigma_i)_{i=1}^n\in \mathbb{R}^n\colon  \sum_{i=1}^n\sigma_i=0 \Big\}.
\end{equation*}
Denote the space of \emph{potential functions} on $I$ by $\mathcal{F}(I)=\mathbb{R}^{n}$, and consider the quotient space 
\begin{equation*}
\mathcal{F}(I)/ \mathbb{R}=\{[\Phi]\mid (\Phi_i)_{i=1}^n\in \mathbb{R}^n\},
\end{equation*}
where $[\Phi]=\{(\Phi_1+c,\cdots, \Phi_n+c)\colon c\in\mathbb{R}\}$ are functions defined up to addition of constants. 

We introduce an identification map via the weighted Laplacian matrix $L(p)$:
\begin{equation*}
\V\colon\mathcal{F}(I)/\mathbb{R} \rightarrow T_p\mathcal{P}_+(I),\quad\quad \V_\Phi=L(p)\Phi.
\end{equation*} 
We know that $L(p)$ has only {one} simple zero eigenvalue with eigenvector $c(1,1,\cdots, 1)$, for any $c\in \mathbb{R}$. This is true since for $(\Phi_i)_{i=1}^n\in\mathbb{R}^n$,
\begin{equation*}
\Phi^{\ts}L(p)\Phi=(D\Phi)^{\ts}\Lambda(p)(D\Phi)=\sum_{(i,j)\in E}\omega_{ij}(\Phi_i-\Phi_j)^2(\frac{1}{2}(\frac{1}{d_i}p_i+\frac{1}{d_j}p_j))=0,
\end{equation*}
implies $\Phi_i=\Phi_j$, $(i,j)\in E$. 
It the graph is connected, as we assume, then $(\Phi_i)_{i=1}^n$ is a constant vector. 
Thus $V_\Phi\colon \mathcal{F}(I)/\mathbb{R}\rightarrow T_p\mathcal{P}_+(I)$ is a well defined map,
linear, and one to one. 
I.e., $\mathcal{F}(I)/\mathbb{R}\cong T_p^\ast\mathcal{P}_+(I)$, where $T_p^\ast\mathcal{P}_+(I)$ is the cotangent space of $\mathcal{P}_+(I)$.
This identification induces the following inner product on 
$T_p\mathcal{P}_+(I)$. 

\begin{definition}[$L^2$-Wasserstein metric tensor]\label{d9}
	The inner product 
	$g_p :T_p\mathcal{P}_+(I)\times T_p\mathcal{P}_+(I)  \rightarrow \mathbb{R}$ 
	takes any two tangent vectors $\sigma_1=\V_{\Phi_1}$ and $\sigma_2=\V_{\Phi_2}\in T_p\mathcal{P}_+(I)$ to 
	\begin{equation}\begin{split}\label{formula}
	g_p(\sigma_1, \sigma_2)=\sigma_1^{\ts}\Phi_2=\sigma_2^{\ts}\Phi_1=\Phi_1^{\ts}L(p)\Phi_2.
	\end{split} 
	\end{equation}
	In other words, 
	\begin{equation*}
	g_p(\sigma_1,\sigma_2):={\sigma_1}^{\ts}L(p)^{\dagger}\sigma_2,\quad \text{for any $\sigma_1,\sigma_2\in T_p\mathcal{P}_+(I)$}, 
	\end{equation*}
	where $L(p)^{\dd}$ is the pseudo inverse of $L(p)$. 
\end{definition}
Following the inner product 
\eqref{formula}, 
the Wasserstein metric (distance function)  $W\colon\mathcal{P}_+(I)\times\mathcal{P}_+(I)\rightarrow\mathbb{R}$ is defined by 
\begin{equation}
W(p^0,p^1)^2:= \inf_{p(t), \Phi(t)}~\Big\{\int_0^1\Phi(t)^{\ts}L(p(t))\Phi(t) dt\Big\}\ .
\label{eq:wm}
\end{equation}
Here the infimum is taken over pairs $(p(t), \Phi(t))$ with 
$p\in H^1((0,1), \mathbb{R}^{n})$ and $\Phi\colon
[0, 1]\rightarrow\mathbb{R}^n$ measurable, satisfying 
\begin{equation*}
\frac{d}{dt}p(t)-L(p(t))\Phi(t)=0,\quad p(0)=p^0,\quad p(1)=p^1. 
\end{equation*}
\subsection{Wasserstein statistical manifold}
We next consider a statistical model defined by a triplet $(\Theta, I, \p)$. 
Here, $I=\{1,\cdots, n\}$ is the sample space, 
$\Theta$ is the parameter space, which is an open subset of $\mathbb{R}^d$, $d\leq n-1$, 
and $\p\colon \Theta\rightarrow \mathcal{P}_+(I)$ is the parametrization function, 
\begin{equation*}
 p(\theta)=(p_i(\theta))_{i=1}^n,\quad \theta\in \Theta. 
\end{equation*}

We define a Riemannian metric $g^W$ on $\Theta$ as the pull-back of metric $g$ on $\mathcal{P}_+(I)$. 
In other words, we require that $\p\colon (\Theta,g^W)\rightarrow (\mathcal{P}_+(I), g)$ is an isometric embedding: 
\begin{equation*}
\begin{split}
g^W_\theta(a,b):=&g^W_{ p(\theta)}(d_\theta p(\theta)(a), d_\theta p(\theta)(b))\\
=&\big(d_\theta p(\theta)(a)\big)^{\ts}L( p(\theta))^{\dd}\big(d_\theta p(\theta)(b)\big),\quad \text{for all $a,b\in T_\theta(\Theta)$}.
\end{split}
\end{equation*}
Since $d p(\theta)(a)=\big(\sum_{j=1}^{d}\frac{\partial p_i(\theta)}{\partial\theta_j}a_j\big)_{i=1}^n=J_\theta p(\theta)a$, 
we arrive at the following definition.
\begin{definition}[$L^2$-Wasserstein metric tensor on parameter space]
	For any pair of tangent vectors $a,b\in T_\theta \Theta=\mathbb{R}^d$, define 
	\begin{equation}\label{GW}
	G_W(\theta):=J_\theta  p(\theta)^{\ts}L( p(\theta))^{\dd}J_\theta  p(\theta),
	\end{equation}
	and 
	\begin{equation*}
	g^W_\theta(a,b):=a^{\ts}G_W(\theta)b,
	\end{equation*}
	where $J_\theta( p(\theta))=(\frac{\partial  p_i(\theta)}{\partial \theta_j})_{1\leq i\leq n, 1\leq j\leq d}\in \mathbb{R}^{n\times d}$ is the Jacobi matrix of the parametrization~$\p$.
\end{definition}
This inner product is consistent with the restriction of the Wasserstein metric $g^W$ to $ p(\theta)$. 
We will assume that $\operatorname{rank}(J_{\theta} p(\theta))=d$, so that the parametrization $\p_i$ is locally injective and the metric tensor $g^W$ is positive definite. 
We call $(\Theta, I, \p)$, together with the induced Riemannian metric $g^W$, \emph{Wasserstein statistical manifold} (WSM). 

In this case, the constrained Wasserstein {distance function} $d_W\colon \Theta\times\Theta \rightarrow \mathbb{R}_+$ is given by the geometric action energy  
\begin{equation}
d_W(\theta_0, \theta_1)^2=\inf_{\theta(t)\in C^1([0,1], \Theta)}\Big\{\int_0^1 \dot\theta(t)^{\ts}G_W(\theta(t))\dot\theta(t) dt\colon \theta(0)=\theta_0,~\theta(1)=\theta_1\Big\}.
\label{eq:wmc}
\end{equation}
{When working on the full probability simplex, with $\theta = p$, the metric function 
$d_W$ corresponds precisely to the metric function $W$ given in~\eqref{eq:wm}. }

\subsection{Fokker-Planck equation on parameter space}
We next derive the Fokker-Planck equation on parameter space by Wasserstein gradient flow of KL divergence. 

Given a reference measure $q\in \mathcal{P}_+(I)$,  consider the Kullback-Leibler divergence (relative entropy) on parameter space
 \begin{equation*}
\operatorname{D}_{\operatorname{KL}}( p(\theta)\| q)=\sum_{i=1}^n  p_i(\theta)\log \frac{ p_i(\theta)}{q_i}.
 \end{equation*}
\begin{proposition}[Fokker-Planck equation on parameter space]
The gradient flow for the negative Boltzmann-Shannon entropy in $(\Theta, g)$ is  
\begin{equation}\label{FPE2}
\frac{d\theta}{dt}=-\big(J_\theta  p(\theta)^{\ts}L( p(\theta))^{\dd}J_\theta  p(\theta)\big)^{\dd}J_\theta  p(\theta)^{\ts}\log \frac{ p(\theta)}{q}.
\end{equation}
\end{proposition}

\begin{proof}
The gradient flow of entropy on $(\Theta, g_W)$ satisfies 
\begin{equation*}
\begin{split}
\frac{d\theta}{dt}=&-\operatorname{grad}_W\operatorname{D}_{\operatorname{KL}}( p(\theta)\| q)\\
=&-G_W(\theta)^{\dd}\nabla_\theta \operatorname{D}_{\operatorname{KL}}( p(\theta)\|q)\\
=&-\big(J_\theta  p(\theta)^{\ts}L( p(\theta))^{\dd}J_\theta  p(\theta)\big)^{\dd}J_\theta  p(\theta)^{\ts}(\log \frac{ p(\theta)}{q}+\vec{1}),
\end{split}
\end{equation*}
where $\nabla_\theta$ represents the Euclidean gradient operator, 
$\log \frac{ p(\theta)}{q}=(\log \frac{ p_i(\theta)}{q_i})_{i=1}^n$ and $\vec{1}=(1,\cdots, 1)\in\mathbb{R}^n$. Since $ p(\theta)^{\ts}\vec{1}=1$, we have $J_\theta p(\theta)^{\ts}\vec{1}=\vec{0}$. This completes the proof. 
\end{proof}
\begin{remark}
Consider the full probability set with continuous sample space $\Omega$. Denote the probability $p_i(t)$ by a density $\rho(t,x)\in \mathcal{P}(\Omega)$, then \eqref{FPE2} recovers the FPE \eqref{FPE}. 
\end{remark}

We next study the convergence properties of the Fokker-Planck equation on parameter space. In other words, how fast does the solution of \eqref{FPE2} converge to its equilibrium? As in the full probability space, we define the concept of Ricci curvature lower bound on parameter space to give the bound of the convergence rate for \eqref{FPE2}. 

\section{Ricci curvature lower bound on parameter space}
\label{sec4}

This section contains the main contributions of this paper. 
We define the Ricci curvature lower bound on parameter space and prove equivalent conditions for this definition, which connect information geometry and Wasserstein geometry. 
In addition, we present several information functional inequalities on parameter space. 
Finally, we give a simple guide for computing these quantities in practice. 

\subsection{Ricci curvature lower bound on parameter space}

\begin{definition}\label{def}
We say $(\Theta, I, \p)$ has the Ricci curvature lower bound $\kappa\in \mathbb{R}$ if for any constant speed geodesic $\theta_t$, $t\in[0,1]$, connecting $\theta_0$, $\theta_1$ in $(\Theta, g^W)$, it holds that 
\begin{equation*}
\operatorname{D}_{\operatorname{KL}}( p(\theta_t)\|q) \leq (1-t)\operatorname{D}_{\operatorname{KL}}( p(\theta_0)\|q) +t\operatorname{D}_{\operatorname{KL}}( p(\theta_1)\|q)-\frac{\kappa}{2}t(1-t)d_W(\theta_0,\theta_1)^2.
\end{equation*}
In this case we also write 
\begin{equation*}
\operatorname{Ric}(\Theta, I, \p)\geq \kappa.
\end{equation*}

If $(\Theta, g_W)$ forms a compact smooth Riemannian manifold and $ p(\theta)$ is smooth. Then $\kappa$ is the smallest eigenvalue of the Hessian of the KL divergence over the Wasserstein statistical manifold, i.e.
\begin{equation*}
\operatorname{Hess}_W\operatorname{D}_{\operatorname{KL}}( p(\theta)\|q)\succeq \kappa G_W(\theta),
\end{equation*}
for any $\theta\in\Theta$. 
\end{definition}
{Definition \ref{def} is based on the definition of geodesic convexity in geometry. 
It is a more general definition than the one in terms of the Hessian operator begin bounded below by $\kappa$.  
The reason is as follows. 
On the one hand,  the probability set is a manifold with boundary. 
Suitable regularity studies are needed to take care of the boundary when using 
the Hessian operator \cite{LiG}. 
On the other hand, not all parameterizations $ p(\theta)$ are twice differentiable. 
} 
 
Definition~\ref{def} shares the same spirit of Lott-Strum-Villani and Erbar-Maas. 
If $ p(\theta)=\mathcal{P}(I)$ is the whole probability simplex, then $\operatorname{Ric}(\Theta, I, \p)$ is the Ricci curvature bound on discrete sample space. 
Our definition extends this idea to a statistical manifold. 
In other words, $\operatorname{Ric}(\Theta, I, \p)$ inherits properties from both probability submanifold and Ricci curvature bound on sample space. Note that $\operatorname{Ric}(\Theta, I, \p)$ is different from the Ricci curvature on $(\Theta, g^W)$, in which the former represents how the changes ratio KL divergence takes effect on the parameterized sample space, while the later reflects the curvature on the set of probability itself.

 We next given an equivalent condition for Definition~\ref{def}. It naturally connects Ricci curvature (R), Information geometry (I) and Wasserstein geometry (W). We call it Ricci-Information-Wasserstein (RIW) condition. 
 
\begin{theorem}[RIW condition]\label{RIWthm}
Assume $\Theta$ is a compact set. $\operatorname{Ric}(\Theta, I, \p)\geq \kappa$ holds if and only if for any $\theta\in \Theta$,
\begin{equation}\label{RIW}
G_F(\theta)+\sum_{a\in I}\Big(d_{\theta\theta}  p_a(\theta)\log\frac{p_a(\theta)}{q}- \Gamma^{W,a}(\theta)\frac{d}{d_{\theta_a}} \operatorname{D}_{\operatorname{KL}}( p(\theta)\|q)\Big)\succeq \kappa G_W(\theta),
\end{equation}
where $G_F(\theta)=(G_F(\theta)_{ab})_{1\leq a,b\leq d}$ is the Fisher-Rao metric tensor 
\begin{equation}\label{Fisher-Rao}
G_{F}(\theta)_{ab}=\sum_{i\in I}\frac{d\log p_i(\theta)}{d\theta_a}\frac{d\log p_i(\theta)}{d\theta_b} p_i(\theta),
\end{equation}
$\Gamma^{W,k}=(\Gamma^{W,k}_{ij})_{1\leq i, j\leq n}$ is the Wasserstein Christoffel symbol
\begin{equation*}
\Gamma^{W,k}_{ij}=\frac{1}{2}\sum_{l=1}^n(G_{W,kl})^{-1}\Big(\nabla_{\theta_{i}}G_{W,jl}+\nabla_{\theta_{j}}G_{W,il}-\nabla_{\theta_{l}}G_{W,ij}\Big),
\end{equation*}
and $G_W(\theta)$ is the Wasserstein metric tensor defined in \eqref{GW}. 
\end{theorem}

\begin{proof}
Let $\theta_t$ be a constant speed geodesic, i.e. $\ddot\theta_t+\Gamma^W(\dot\theta_t, \dot\theta_t)=0$ with $\theta_0=\theta\in \Theta$ and $\dot\theta_0=a\in T_\theta\Theta$. 
Consider the Taylor expansion
$$
\operatorname{D}_{\operatorname{KL}}( p(\theta_t)\|q)
=\operatorname{D}_{\operatorname{KL}}( p(\theta)\|q)
+\left.\frac{d}{dt}\right|_{t=0}\operatorname{D}_{\operatorname{KL}}( p(\theta_t)\|q)t
+\frac{1}{2}\left.\frac{d^2}{dt^2}\right|_{t=0}\operatorname{D}_{\operatorname{KL}}( p(\theta_t)\|q)t^2+o(t^2). 
$$
Then the Hessian operator on Riemannian manifold $(\Theta, g^W)$ 
forms 
\begin{equation*}
\begin{split}
\operatorname{Hess}_W\operatorname{D}_{\operatorname{KL}}( p(\theta_t)\|q)(\dot\theta_t,\dot\theta_t)
=&\frac{d^2}{dt^2}D_\textrm{KL}( p(\theta_t)\|q)\\
=&\frac{d}{dt}(d_\theta D_\textrm{KL}( p(\theta_t)\|q)^{\ts}\dot\theta_t)\\
=&\dot\theta_t^{\ts}d_{\theta\theta}\operatorname{D}_{\operatorname{KL}}( p(\theta_t)\|q)\dot\theta_t-d_\theta \operatorname{D}_{\operatorname{KL}}( p(\theta_t)\|q)^{\ts}\Gamma^W(\dot\theta_t, \dot\theta_t)\\
=&\dot\theta_t^{\ts}d_{\theta\theta}\operatorname{D}_{\operatorname{KL}}( p(\theta_t)\|q)\dot\theta_t-\dot\theta_t^{\ts}(\sum_{k\in I}\frac{d}{d\theta_k} \operatorname{D}_{\operatorname{KL}}( p(\theta_t)\|q)\Gamma^{W,k})\dot\theta_t.
\end{split}
\end{equation*}
In addition, 
\begin{equation*}
\begin{split}
d_\theta\operatorname{D}_{\operatorname{KL}}( p(\theta_t)\|q)=&\sum_{i=1}^n\Big(d_\theta p_i(\theta_t)\log p_i(\theta_t)+ p_i(\theta_t)d_\theta\log p_i(\theta_t)-d_\theta p_i(\theta_t)\log q_i\Big)\\
=&\sum_{i=1}^n\Big(d_\theta p_i(\theta_t)\log p_i(\theta_t)-d_\theta p_i(\theta_t)\log q_i\Big),
\end{split}
\end{equation*}
where $\sum_{i=1}^n p_i(\theta_t)d_\theta\log p_i(\theta_t)=\sum_{i=1}^n p_i(\theta_t)\frac{1}{ p_i(\theta_t)}J_\theta p_i(\theta_t)=0$, since $\sum_{i=1}^np_i(\theta)=1$.
Thus
\begin{equation*}
\begin{split}
d_{\theta\theta}\operatorname{D}_{\operatorname{KL}}( p(\theta_t)\|q)=&\sum_{i=1}^nd_{\theta\theta} p_i(\theta_t)\log\frac{ p_i(\theta_t)}{q}+\sum_{i=1}^n\frac{1}{ p_i(\theta_t)}d_\theta  p_i(\theta_t)d_\theta p_i(\theta_t)^{\ts}\\
=&\sum_{i=1}^nd_{\theta\theta} p_i(\theta_t)\log\frac{ p_i(\theta_t)}{q}+G_F(\theta),
\end{split}
\end{equation*}
where $G_F$ denotes the Fisher-Rao metric tensor: 
$G_F(\theta_t)=\sum_{i=1}^n\frac{1}{ p_i(\theta_t)}d_\theta  p_i(\theta_t)d_\theta p_i(\theta_t)^{\ts}=\sum_{i=1}^nd_\theta\log p_i(\theta_t)d_\theta\log p_i(\theta_t)^{\ts} p_i(\theta)$, with the fact $\frac{1}{ p_i(\theta_t)}d_\theta  p_i(\theta_t)=d_\theta\log p_i(\theta_t)$.

Thus $\operatorname{Hess}_W\operatorname{D}_{\operatorname{KL}}( p(\theta)\|q)\succeq \kappa G_W(\theta)$ is equivalent to~\eqref{RIW}. This concludes the proof.
\end{proof}
\begin{remark}
If we replace $I$ by the continuous sample space $(\Omega, g^\Omega)$ and consider the full probability simplex, 
then RIW condition~\eqref{RIW} is equivalent to the integral version of Bakery-Emery condition. 
See details in~\cite[Proposition~19]{LiG}.
\end{remark}

\subsection{Entropy dissipation on parameter space}
With the Ricci curvature lower bound in hand, we can prove the following convergence properties of Fokker-Planck equations on parameter space. 

\begin{proposition}[Bakery-Emery condition on parameter space]
	\label{thm}
Assume $\Theta$ is a compact set. If $\operatorname{Ric}(\Theta, I, \p)\geq \kappa>0$, then there exists a unique equilibrium $\theta^\ast\in \Theta$, with 
\begin{equation*}
\theta^\ast=\arg\min_{\theta\in \Theta}\operatorname{D}_{\operatorname{KL}}( p(\theta)\|q).
\end{equation*}
In addition, for any initial condition $\theta_0\in \Theta$, the solution $\theta(t)$ of \eqref{FPE} converges to $\theta^\ast$ exponentially fast, with 
\begin{equation}\label{expon}
\operatorname{D}_{\operatorname{KL}}( p(\theta_t)\|q)-\operatorname{D}_{\operatorname{KL}}( p(\theta^\ast)\|q)\leq e^{-2\kappa t}\Big(\operatorname{D}_{\operatorname{KL}}( p(\theta_0)\|q)-\operatorname{D}_{\operatorname{KL}}( p(\theta^\ast)\|q)\Big),\quad \text{for all $t$}.
\end{equation}
\end{proposition}
\begin{remark}{This result will apply for any geometry defined on $\Theta$, whenever $\kappa$ is the smallest eigenvalue of the corresponding Hessian operator of the divergence function.}
\end{remark}
\begin{proof}
The proof comes from the classical study of gradient flow in Riemannian manifold $(\Theta, g_W)$. Since $\operatorname{Hess}_W\operatorname{D}_{\operatorname{KL}}( p(\theta)\|q)\geq \kappa>0$, the $\operatorname{D}_{\operatorname{KL}}( p(\theta)\|q)$ is $\kappa$-geodesics convex in $(\Theta, g_W)$. Thus $\theta(t)$ converges to the unique equilibrium $\theta^\ast$, which is also the unique minimizer of KL divergence. 

We next investigate how fast $\theta(t)$ converges to $\theta^\ast$. 
The speed of convergence is obtained by comparing 
the first and second derivatives of the KL divergence w.r.t.~time $t$ along~\eqref{FPE}. 
We have 
\begin{equation*}
\frac{d}{dt}\operatorname{D}_{\operatorname{KL}}( p(\theta_t)\|q)=-g_W(\operatorname{grad}_W\operatorname{D}_{\operatorname{KL}}( p(\theta_t)\|q), \operatorname{grad}_W\operatorname{D}_{\operatorname{KL}}( p(\theta_t)\|q)\big),
\end{equation*}
and 
\begin{equation*}
\frac{d^2}{dt^2}\operatorname{D}_{\operatorname{KL}}( p(\theta_t)\|q)=2\operatorname{Hess}_W\operatorname{D}_{\operatorname{KL}}( p(\theta_t)\|q)\big(\operatorname{grad}_W\operatorname{D}_{\operatorname{KL}}( p(\theta_t)\|q), \operatorname{grad}_W\operatorname{D}_{\operatorname{KL}}( p(\theta_t)\|q)\big).
\end{equation*}
From $\operatorname{Ric}(\Theta, I, \p)\geq \kappa>0$, then $\operatorname{Hess}_W\operatorname{D}_{\operatorname{KL}}( p(\theta)\|q)\geq \kappa>0$, i.e. 
\begin{equation}\label{C1} 
\frac{d^2}{dt^2}\operatorname{D}_{\operatorname{KL}}( p(\theta_t)\|q) \geq -2\kappa\frac{d}{dt}\operatorname{D}_{\operatorname{KL}}( p(\theta_t)\|q),\quad\text{for all  $t\geq 0$}.
\end{equation}
Then by integrating the above formula over $[t,+\infty)$, one obtains 
\begin{equation*}
\frac{d}{dt}[\operatorname{D}_{\operatorname{KL}}( p(\theta^\ast)\|q)-\operatorname{D}_{\operatorname{KL}}( p(\theta_t)\|q)]\geq -2\kappa[\operatorname{D}_{\operatorname{KL}}( p(\theta^\ast)\|q)-\operatorname{D}_{\operatorname{KL}}( p(\theta_t)\|q)].
\end{equation*}
Proceed with the Gr\"onwall's inequality, the result is proved. 
\end{proof}
\subsection{Functional inequalities on parameter space}
In literature \cite{OV}, the convergence rate of FPE is used to prove several functional inequalities, including Log-Sobolev, Talagrand and HWI inequalities. 
The HWI inequality is a relation between the relative entropy (H), Wasserstein metric (W), relative Fisher information functional ($\mathcal{I}$). 
We shall derive the counterparts of these inequalities on parameter space.  

Here the Log-Sobolev inequality describes a relationship between relative entropy and relative Fisher information functional on parameter space. Here the relative Fisher information functional is defined by
\begin{equation}\label{RF}
\begin{split}
\mathcal{I}( p(\theta)\|q):=&g_W(\operatorname{grad}_W\operatorname{D}_{\operatorname{KL}}( p(\theta_t)\|q), \operatorname{grad}_W\operatorname{D}_{\operatorname{KL}}( p(\theta_t)\|q)).
\end{split}
\end{equation}
 In particular, we formulate \eqref{RF} as follows:
\begin{equation*}
\begin{split}
\mathcal{I}( p(\theta)\|q)=&\log \frac{ p(\theta)}{q}^{\ts}J_\theta( p(\theta))\big(J_\theta  p(\theta)^{\ts}L( p(\theta))^{\dd}J_\theta  p(\theta)\big)^{\dd}J_\theta  p(\theta)^{\ts}\log \frac{ p(\theta)}{q}\\
=&\Big(\textrm{Proj}_{\theta}\log \frac{ p(\theta)}{q}\Big)^{\ts}L( p(\theta))\Big(\textrm{Proj}_\theta\frac{ p(\theta)}{q}\Big)\\
=& \sum_{i=1}^n \sum_{j\in N(i)}\frac{1}{2d_i}\omega_{ij}\Big((\textrm{Proj}_{\theta}\log \frac{ p(\theta)}{q})_i-(\textrm{Proj}_{\theta}\log\frac{ p(\theta)}{q})_j\Big)^2 p_i(\theta),
\end{split}
\end{equation*}
where $\textrm{Proj}_\theta=(J_\theta p(\theta)^{\ts})^{\dd}J_\theta  p(\theta)^{\ts}\in \mathbb{R}^{n\times n}$ is the projection matrix, which projects the differential operator in full probability into the one in parameter space. We compare \eqref{RF} with the one in continuous sample space and full probability space: 
\begin{equation*}
\mathcal{I}(\rho\|q)=\int_{\Omega}g^\Omega(\nabla \log\frac{\rho}{q}, \nabla \log\frac{\rho}{q})\rho \;dx. 
\end{equation*}
We note that the functional \eqref{RF} is different from the commonly known Fisher information matrix \eqref{Fisher-Rao} in parameter space. It contains the ground metric structure in the sample space, which is inherited from the $L^2$-Wasserstein metric tensor $L(p)^{\dd}$. In other words, when applying the Fisher information in full probability set into parameter space, the following two angles arrive. {Here \eqref{RF} keeps the differential structure of sample space and project the differential of KL divergence into the parameter space, while Fisher information matrix \eqref{Fisher-Rao} replaces the differential structures of sample space to the ones in parameters.}

In the following, we derive inequalities based on \eqref{RF}.
\begin{proposition}[Functional inequalities on parameter space]\label{FI}
Consider a statistical manifold $(\Theta, I, \p)$. The following inequalities hold. 
\begin{itemize}
\item[(i)]If $\operatorname{Ric}(\Theta, I, \p)\geq \kappa>0$, then the Logarithmic Sobolev inequality on parameter space
\begin{equation}\label{LS}
\operatorname{D}_{\operatorname{KL}}( p(\theta)\|q)-\operatorname{D}_{\operatorname{KL}}( p(\theta^\ast)\|q)  \leq \frac{1}{2\kappa}\mathcal{I}( p(\theta)\|q) ,
\end{equation}
holds for any $\theta\in \Theta$. 
\item[(ii)] If $\operatorname{Ric}(\Theta, I, \p)\geq \kappa>0$, then the Talagrand inequality on parameter space
\begin{equation*}
\frac{\kappa}{2}d_W(\theta, \theta^\ast)^2\leq {\operatorname{D}_{\operatorname{KL}}( p(\theta)\|q)-\operatorname{D}_{\operatorname{KL}}( p(\theta^\ast)\|q)} ,
\end{equation*}
holds for any $\theta\in \Theta$. 
\item[(iii)] If $\operatorname{Ric}(\Theta, I, \p)\geq \kappa\in\mathbb{R}$ ($\kappa$ not necessarily positive), then the HWI inequality on parameter space
\begin{equation*}
\operatorname{D}_{\operatorname{KL}}( p(\theta)\|q)-\operatorname{D}_{\operatorname{KL}}( p(\theta^\ast)\|q)\leq \sqrt{\mathcal{I}( p(\theta)\|q)}d_W(\theta,\theta^\ast)-\frac{\kappa}{2}d_W(\theta, \theta^\ast)^2,
\end{equation*}
holds for any $\theta\in \Theta$. 
\end{itemize}
\end{proposition}
\begin{proof}
Here we mainly follow the heuristic arguments in \cite{OV}. In finite dimensional parameter space, these approaches are rigorous. We demonstrate the proofs for the completeness of paper. 

(i)~The proof follows Proposition \ref{thm}. Consider the Fokker-Planck equation \eqref{FPE} with initial condition $\theta(0)=\theta$. 
The dissipation along gradient flow of entropy gives
\begin{equation}\label{aa}
\begin{split}
\mathcal{I}( p(\theta_t)\|q)=&-\frac{d}{dt}\operatorname{D}_{\operatorname{KL}}( p(\theta_t)\|q)\\
=&g_W\big(\operatorname{grad}_W \operatorname{D}_{\operatorname{KL}}( p(\theta_t)\|q), \operatorname{grad}_W \operatorname{D}_{\operatorname{KL}}( p(\theta_t)\|q)\big). 
\end{split}
\end{equation}
Since \eqref{C1} holds, by integrating over time $t\in [0, \infty)$, we have 
 \begin{equation*}
-\frac{d}{dt}\operatorname{D}_{\operatorname{KL}}( p(\theta_t)\|q)|_{t=0}^\infty\geq -2\kappa[\operatorname{D}_{\operatorname{KL}}( p(\theta^\ast)\|q)-\operatorname{D}_{\operatorname{KL}}( p(\theta_0)\|q)].
\end{equation*}
From \eqref{aa} and $\theta(0)=\theta$, we have
\begin{equation*}
\mathcal{I}( p(\theta)\|q)- \mathcal{I}( p(\theta^\ast)\|q)\leq 2\kappa[\operatorname{D}_{\operatorname{KL}}( p(\theta)\|q)-\operatorname{D}_{\operatorname{KL}}( p(\theta^\ast)\|q)],
\end{equation*}
where we use the fact  $\textrm{grad}_W\operatorname{D}_{\operatorname{KL}}( p(\theta)\|q)=0$, so that $\mathcal{I}( p(\theta^\ast)\|q)=0$. 
It proves the result. 

(ii) Consider $\theta(t)$ satisfy the FPE \eqref{FPE} on parameter space with $\theta(0)=\theta$. Since $\operatorname{Ric}(\Theta, I, \p)\geq \kappa>0$, then $\lim_{t\rightarrow \infty}\theta(t)=\theta^\ast$.
Define 
\begin{equation*}
\Psi(t)=d_W(\theta, \theta(t))+\sqrt{\frac{2}{\kappa}}\sqrt{\operatorname{D}_{\operatorname{KL}}( p(\theta_t)\|q)-\operatorname{D}_{\operatorname{KL}}( p(\theta^\ast)\|q)}.
\end{equation*}
Thus $\Psi(0)=\sqrt{\frac{2}{\kappa}}\sqrt{\operatorname{D}_{\operatorname{KL}}( p(\theta)\|q)-\operatorname{D}_{\operatorname{KL}}( p(\theta^\ast)\|q)}$ and $\Psi(\infty)=\lim_{t\rightarrow \infty}\Psi(t)=d_W(\theta, \theta^\ast)$. 
We claim that $\Psi(t)$ is nondecreasing. If so, then $\Psi(0)\leq \Psi(\infty)$, which proves the result. 

To show $\Psi(t)$ is nondecreasing, we shall prove that 
$$
\frac{d}{dt}^+\Psi(t)=\lim\sup_{h\rightarrow 0+}\frac{\Psi(t+h)-\Psi(t)}{h}\leq 0.
$$
Here we assume $\theta(t)\neq\theta^\ast$, otherwise $\Psi(t+h)=\Psi(t)$ for any $h$, which shows the upper derivative zero. 

On the one hand, by triangle inequality,
$$|d_W(\theta, \theta_t)- d_W(\theta, \theta_{t+h})|\leq d_W(\theta_t, \theta_{t+h}),$$
so that
\begin{equation}\label{a}
\lim\sup_{h\rightarrow 0+}\frac{d_W(\theta_t, \theta_{t+h})}{h}=\sqrt{g_W(\operatorname{grad}_W\operatorname{D}_{\operatorname{KL}}( p(\theta_t)\|q),\operatorname{grad}_W\operatorname{D}_{\operatorname{KL}}( p(\theta_t)\|q))}=\sqrt{\mathcal{I}( p(\theta_t)\|q)}.
\end{equation}

On the other hand, since $\theta(t)\neq \theta^\ast$, then 
\begin{equation*}
\begin{split}
&\sqrt{\frac{2}{\kappa}}\frac{d}{dt}\sqrt{\operatorname{D}_{\operatorname{KL}}( p(\theta_t)\|q)-\operatorname{D}_{\operatorname{KL}}( p(\theta^\ast)\|q)}\\
=&-\frac{g_W(\operatorname{grad}_W\operatorname{D}_{\operatorname{KL}}( p(\theta_t)\|q),\operatorname{grad}_W\operatorname{D}_{\operatorname{KL}}( p(\theta_t)\|q))}{\sqrt{2\kappa (\operatorname{D}_{\operatorname{KL}}( p(\theta_t)\|q)-\operatorname{D}_{\operatorname{KL}}( p(\theta^\ast)\|q))}}\\
=&-\frac{\mathcal{I}( p(\theta_t)\|q)}{\sqrt{2\kappa (\operatorname{D}_{\operatorname{KL}}( p(\theta_t)\|q)-\operatorname{D}_{\operatorname{KL}}( p(\theta^\ast)\|q))}}.
\end{split}
\end{equation*}
 From \eqref{LS}, we have 
\begin{equation}\label{b}
\sqrt{\frac{2}{\kappa}}\frac{d}{dt}\sqrt{\operatorname{D}_{\operatorname{KL}}( p(\theta_t)\|q)-\operatorname{D}_{\operatorname{KL}}( p(\theta^\ast)\|q)}\leq  -\sqrt{\mathcal{I}( p(\theta_t)\|q)}.
\end{equation}
From \eqref{a} and \eqref{b}, we have $\frac{d}{dt}^+\Psi(t)=\lim\sup_{h\rightarrow 0+}\frac{\Psi(t+h)-\Psi(t)}{h}\leq 0$, which finishes the proof.

(iii) From the definition of $\operatorname{Ric}(\Theta, I, \p)\geq \kappa$, then $\operatorname{Hess}_W\operatorname{D}_{\operatorname{KL}}( p(\theta)\|q)\succeq  \kappa G_W$. Denote $\theta_t$ be a geodesic curve of least energy in manifold $(\Theta, g_W)$, joining $\theta_0=\theta$ and $\theta_1=\theta^\ast$. Thus 
\begin{equation*}
d_W(\theta, \theta^\ast)=\sqrt{g_W(\frac{d\theta_t}{dt}, \frac{d\theta_t}{dt})}.
\end{equation*}
From the Taylor expansion on the $(\Theta, g_W)$, we have 
\begin{equation*}
\operatorname{D}_{\operatorname{KL}}( p(\theta^\ast)\|q)=\operatorname{D}_{\operatorname{KL}}( p(\theta)\|q)+\frac{d}{dt}|_{t=0}\operatorname{D}_{\operatorname{KL}}( p(\theta_t)\|q)+\int_0^1(1-t)\frac{d^2}{dt^2}\operatorname{D}_{\operatorname{KL}}( p(\theta_t)\|q)dt.
\end{equation*}
We note that 
\begin{equation*}
\begin{split}
\frac{d}{dt}|_{t=0}\operatorname{D}_{\operatorname{KL}}( p(\theta_t)\|q)=&g_W(\operatorname{grad}_W\operatorname{D}_{\operatorname{KL}}( p(\theta_t)\|q), \frac{d\theta_t}{dt})|_{t=0}\\
\geq&-\sqrt{g_W(\operatorname{grad}_W\operatorname{D}_{\operatorname{KL}}( p(\theta_t)\|q), \operatorname{grad}_W\operatorname{D}_{\operatorname{KL}}( p(\theta_t)\|q))}|_{t=0}\sqrt{g_W(\frac{d\theta_t}{dt}, \frac{d\theta_t}{dt})}|_{t=0}\\
=&-\sqrt{\mathcal{I}( p(\theta)\|q)} d_W(\theta, \theta^\ast),
\end{split}
\end{equation*}
and 
\begin{equation*}
\begin{split}
\int_0^1(1-t)\frac{d^2}{dt^2}\operatorname{D}_{\operatorname{KL}}( p(\theta_t)\|q)dt=&\int_0^1(1-t) g_W(\operatorname{Hess}_W\operatorname{D}_{\operatorname{KL}}( p(\theta_t)\|q)\cdot\frac{d\theta_t}{dt}, \frac{d\theta_t}{dt})dt\\
\geq&\int_0^1\kappa(1-t)  g_W(\frac{d\theta_t}{dt}, \frac{d\theta_t}{dt})dt \\
=&\frac{\kappa}{2}d_W(\theta, \theta^\ast)^2.
\end{split}
\end{equation*}
Combining the above formulas, we prove the result.
\end{proof}

\subsection{Computing the Ricci curvature lower bound and convergence rate}

In this section, we design an algorithm for Ricci curvature lower bound $\kappa$.

We first approximate $\kappa$ by RIW condition in Theorem~\ref{RIWthm}. In other words, we compute formulas for \eqref{RIW} via 
\begin{equation*}
\kappa=\textrm{Smallest eigenvalue of }~G_W(\theta)^{-1}\Big\{G_F(\theta)+\sum_{a\in I}\Big(d_{\theta\theta}  p_a(\theta)\log\frac{p_a(\theta)}{q}- \Gamma^{W,a}(\theta)\frac{d}{d_{\theta_a}} \operatorname{D}_{\operatorname{KL}}( p(\theta)\|q)\Big)\Big\}.
\end{equation*}
where $d_{\theta\theta}  p_a(\theta)$, $\Gamma^{W,a}(\theta)$, $\frac{d}{d{\theta_a}} \operatorname{D}_{\operatorname{KL}}( p(\theta)\|q)$ are computed by numerical differentiation. 

{In practice, we also compute {a uniform convergence rate} $K\geq \kappa$ as the smallest ratio of $\frac{d}{dt}\operatorname{D}_{\operatorname{KL}}( p(\theta_{t})\|q)$ and 
$\frac{d^2}{dt^2}\operatorname{D}_{\operatorname{KL}}( p(\theta_{t})\|q)$ along the gradient flow \eqref{FPE2} for any initial conditions. I.e.
\begin{equation*}
K=\min_{\theta_0\in \Theta}\frac{1}{2T}\frac{\operatorname{D}_{\operatorname{KL}}( p(\theta_{2T})\|q)- 2\operatorname{D}_{\operatorname{KL}}( p(\theta_{T})\|q)+\operatorname{D}_{\operatorname{KL}}( p(\theta_0)\|q)}{\operatorname{D}_{\operatorname{KL}}( p(\theta_T)\|q)- \operatorname{D}_{\operatorname{KL}}( p(\theta_0)\|q)}, 
\end{equation*}
where $T$ is a given short time, $\theta_T$ is the solution of \eqref{expon} with initial condition $\theta_0$.} Whenever $K>0$, it is always the tight bound for functional inequalities in Proposition \ref{FI}. 
\begin{table}[H]\label{Alg}
\begin{tabbing}
aaaaa\= aaa \=aaa\=aaa\=aaa\=aaa=aaa\kill  
   \rule{\linewidth}{0.8pt}\\
   \noindent{\large\bf Convergence rate}\\
     \1 \textbf{Input}: Sample initial conditions $\{\theta_0^s\}_{s=1}^{|S|}$; \\
    \3 Target distribution $q$;\\ 
  \3 A suitable initial step size $h>0$;\\
  \3 A short terminal time $T>0$.\\
  \1 \textbf{Output}: 
  Approximation {$K$} of the uniform convergence rate;\\
    \rule{\linewidth}{0.8pt}\\ 
  \1 \For $s\in \{1,\cdots, |S|\}$\\
  \2 \For $k=1, 2, \ldots, 2T/h$\\
 \3 $\theta^s_{k+1} = \theta^s_k-hG_W(\theta_k^s)^{-1}\nabla_{\theta}\operatorname{D}_{\operatorname{KL}}(p(\theta_k^s)\|q)$ \ ; \\
  \2 \End\\
  \1 \End\\
\1  $K=\min_{s\in \{1,\cdots, |S|\}}\frac{1}{2T}\frac{\operatorname{D}_{\operatorname{KL}}( p(\theta^s_{2T})\|q)- 2\operatorname{D}_{\operatorname{KL}}( p(\theta^s_{T})\|q)+\operatorname{D}_{\operatorname{KL}}( p(\theta^s_0)\|q)}{\operatorname{D}_{\operatorname{KL}}( p(\theta^s_T)\|q)- \operatorname{D}_{\operatorname{KL}}( p(\theta^s_0)\|q)}$.
\\
   \rule{\linewidth}{0.8pt}
\end{tabbing}
\end{table}
\section{Examples}
\label{sec5}

In this section, we illustrate some of the concepts introduced in the previous sections by means of evaluating them on 
a simple class of exponential family models. We illustrate the effects from the choice of the ground metric on sample space in relation to the choice of the statistical model, 
and the relationships between the Ricci curvature lower bound and the rates of convergence in learning. 

\begin{example}[Ricci curvature for a one-dimensional exponential family on three states]
	\label{ex1}
We study how the Ricci curvature changes with the choice of a probability model and with the choice of the ground metric on sample space. 
In order to obtain a picture as complete as possible, we consider the small setting of three states and one dimensional exponential families.   

Consider the sample space $I=\{1,2,3\}$ with a 
fully connected graph with edges $E=\{ (1,2), (2,3),(1,3)\} $, and weights $\omega=(\omega_{12},\omega_{23},\omega_{13})$. The probability simplex is a triangle 
\begin{equation*}
\mathcal{P}(I)=\Big\{(p_i)_{i=1}^3\in \mathbb{R}^3~:~\sum_{i=1}^3p_i=1, \quad p_i\geq 0\Big\}. 
\end{equation*}
We consider statistical manifolds of the form 
\begin{equation*}
 p(\theta)=\frac{1}{Z(\theta)}(e^{\theta c_1}, e^{\theta c_2}, e^{\theta c_3}),
 \end{equation*}
with sufficient statistic $c=(c_1,c_2,c_3)\in\mathbb{R}^3$, 
parameter $\theta\in \Theta=[\theta_{\text{min}},\theta_{\text{max}}]\subset\mathbb{R}^1$, 
and partition function $Z(\theta)=\sum_{i=1}^3e^{\theta c_i}$. These are exponential families specified by the choice of the sufficient statistic $c$. 
Here, addition of constants is immaterial. 
Multiplicative scaling by non-zero numbers does not change the model. 
For better comparability, we always choose $c$ to have norm one. 

In particular, these models can be indexed by the projective line, which for simplicity we can represent by a half circle, or an angle. 

We fix a uniform reference measure $q = (\frac{1}{3},\frac{1}{3}, \frac{1}{3})$. 
The KL divergence then takes the form 
\begin{equation*}
\operatorname{D}_{\operatorname{KL}}(p\|q)=\sum_{i=1}^3p_i\log\frac{p_i}{q_i}=\sum_{i=1}^3p_i\log p_i+\log 3. 
\end{equation*}

We evaluate the Ricci curvature lower bound for 30 different exponential families and 10 different choices of the ground metric. 
We choose the sufficient statistics as evenly spaced points on a radius 1 half circle, and set the parameter domain as $\Theta = [-2,2]$. 

The results are shown in Figure~\ref{fig:experimen1}. The left panel estimates $K$ as the minimum rate of convergence of the Wasserstein gradient flow of the KL divergence, over a grid of 10 different initial conditions on the parameter domain. 
As can be seen, the convergence is faster, the better $\omega$ connects the end points of the exponential family. 

The right panel estimates $\kappa$ as the minimum eigenvalue of the Hessian operator of the KL divergence over a grid of parameter values in the domain. Figure~\ref{fig:experiment1b} gives a direct comparison of the estimates obtained from convergence rates and the Hessian. As can be seen, the Hessian is always a lower bound of the convergence rate, which reflects Proposition~\ref{thm}. 

If the parameter domain is smaller, the Hessian gives a closer bound to the rate of convergence. 
If, on the contrary, the parameter domain is larger, the gaps between the Hessian and the convergence rates tend to be larger. 
Larger parameters correspond to distributions closer to the boundary of the simplex. We illustrate these effects in the Appendix, where we provide figures with different choices of $\Theta$ (Figures~\ref{fig:experiment1s} and~\ref{fig:experiment1l}), and also comparing the Hessian and rates of convergence at individual parameter values (Figure~\ref{fig:experiment1pointwise}).

\begin{figure}
\includegraphics[clip=true, trim=0cm 5.5cm 0cm 5.5cm,width=.49\linewidth]{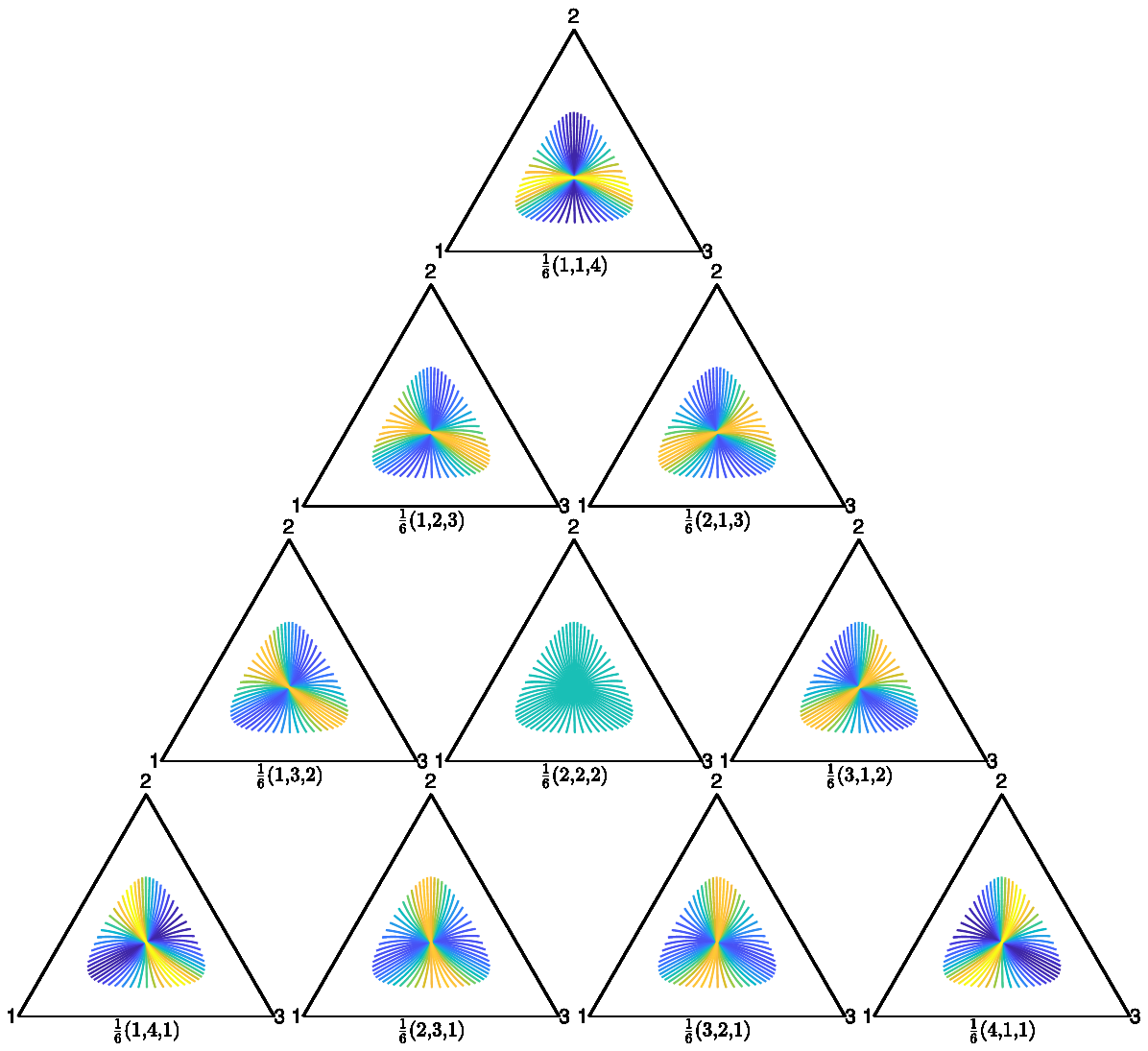}
\includegraphics[clip=true, trim=0cm 5.5cm 0cm 5.5cm,width=.49\linewidth]{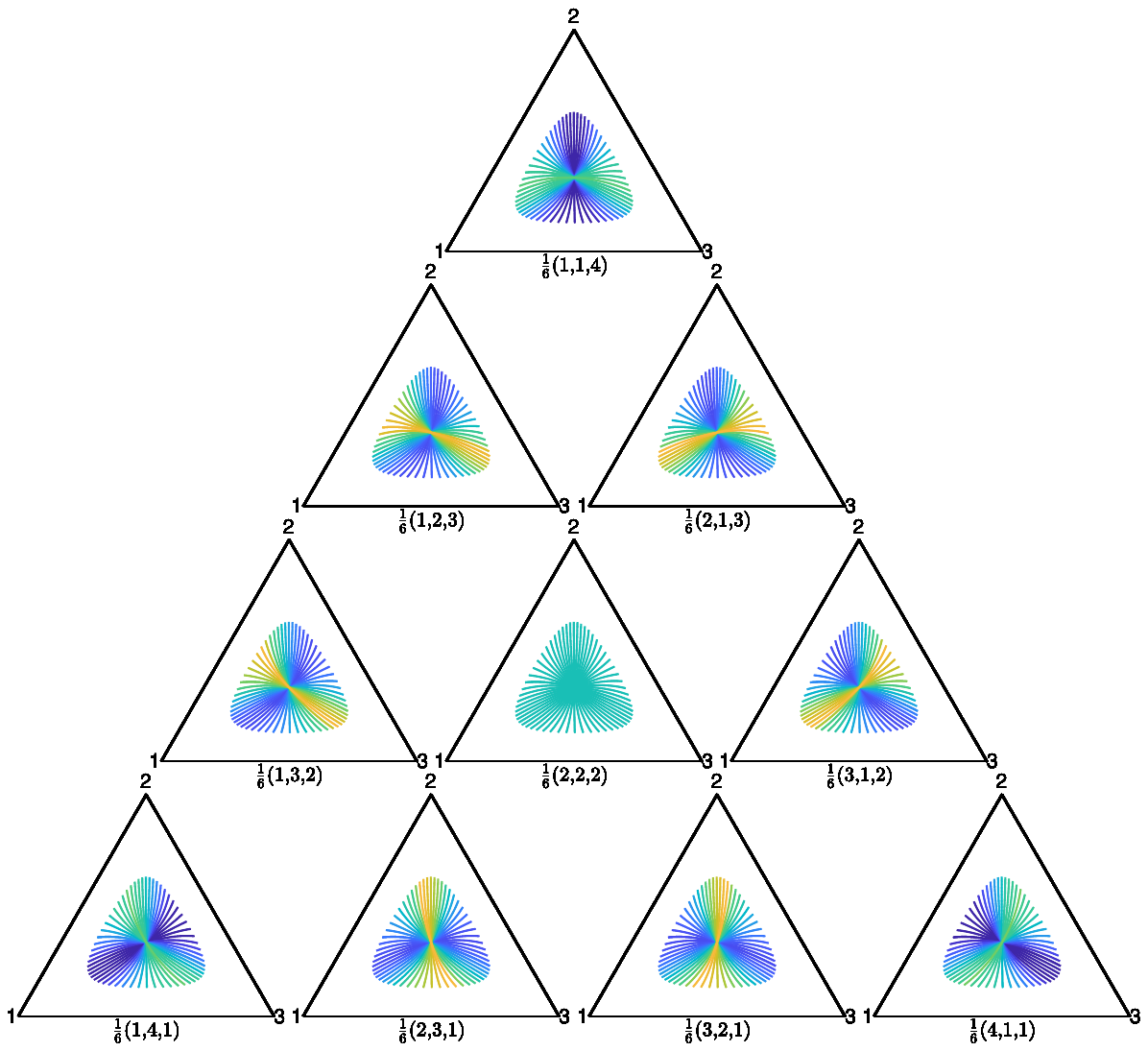}
	\caption{Lower bound on the Ricci curvature for one-dimensional exponential families on three states. 
Each simplex corresponds to a different choice of $\omega=(\omega_{12},\omega_{23},\omega_{13})$, indicated at the bottom. 
Within each simplex there are 30 different exponential families (which are curves) with sufficient statistics of norm one and parameter domain $\Theta=[-1,1]$. 
The color of each exponential family corresponds to the value of $K$ estimated as the minimum convergence rate (left panel), 
and the value of $\kappa$ as the minimum eigenvalue of the Hessian (right panel), over the parameter domain. 
Blue corresponds to lower and yellow to higher values. We give a direct comparison of $K$ and $\kappa$ in Figure~\ref{fig:experiment1b}. 
	}
	\label{fig:experimen1}
\end{figure}
\begin{figure}
\centering
\includegraphics[clip=true, trim=2.5cm 10.5cm 2cm 10cm,width=.8\linewidth]{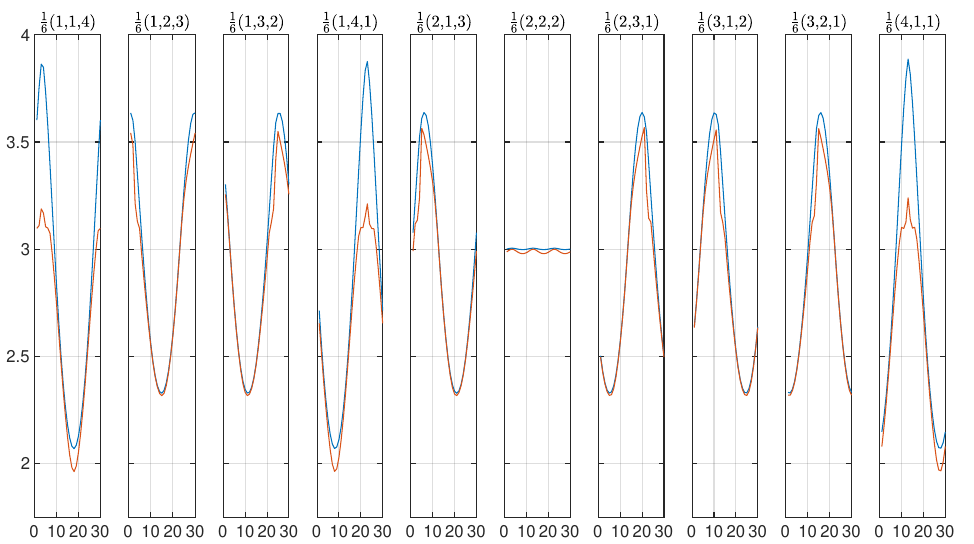}	
\caption{ 
This figure compares the values of $K$ and $\kappa$ from Figure~\ref{fig:experimen1}. Each subplot corresponds to one choice of $\omega$, indicated at the top, with x axis corresponding to the 30 different exponential families. 
As can be seen, the curvature $\kappa$ obtained as the smallest Hessian eigenvalue (red) is, indeed, always a lower bound of the convergence rate $K$ (blue). }
\label{fig:experiment1b}
\end{figure}

\end{example}

\section{Discussion}

To summarize, we introduced a notion of Ricci curvature lower bound for parametric statistical models and illustrated its possible relevance in the context of parameter estimation and learning. 
This notion is based on the geodesic convexity of the KL divergence in Wasserstein geometry. Following the program from~\cite{WSM1}, we hope that this paper continues to strengthen the interactions between information geometry and Wasserstein geometry. 

The Ricci curvature lower bound depends on the target distribution, the statistical model, and the ground metric on sample space. 
We think that this notion can serve to capture the general properties of learning in different models, and hence that it can serve to guide the design of statistical models (e.g., the graph of a graphical model or the connectivity structure of a neural network) and the ground metric. 
Our experiments show that an adequate choice of the two, in conjunction, can significantly increase the rates of convergence in learning. 
On the other hand, the Ricci curvature depends on both, the information and the Wasserstein metric tensors. 
An interesting question arises; namely to find the statistical interpolation of such a connection.

We note that the Ricci curvature lower bound is a global notion over the probability model. 
This is advantageous to provide a uniform analysis, but it can also lead to difficulties, especially when the models include points near the boundary of the simplex, where the behavior is not as regular. 
Our experiments indicate that restricting the parameter domain to a region bounded away from the boundary of the simplex allows us to closely track the rates of convergence. Another challenge is that, being a global quantity, the computation can be challenging. 
Nonetheless, we point out that computing the curvature in terms of the Hessian is much cheaper than estimating the learning rates empirically. We have focused on discrete sample spaces, which allowed us to obtain an intuitive and transparent picture of the relationships that derive form this theory. 
However, we expect that the derivations extend naturally to the case of continuous sample spaces. 
 
Another interesting line of investigation is the following. Our definitions are based on the KL divergence and the Wasserstein and Fisher metric tensors. In principle, it is possible to derive analogous definitions for other metric structures. In particular, one can consider the family of f-divergences. Such an analysis could allow us to compare different learning paradigms.

\appendix 

\newpage

\bigskip

\section{Additional figures to Example~\ref{ex1}}
\begin{figure}[h]
\scalebox{1.2}{ 
	\begin{tikzpicture}[x=1cm,y=1cm]
	\node[] at (0,0){\includegraphics[clip=true, trim=0cm 5.5cm 0cm 5.5cm,width=6cm]{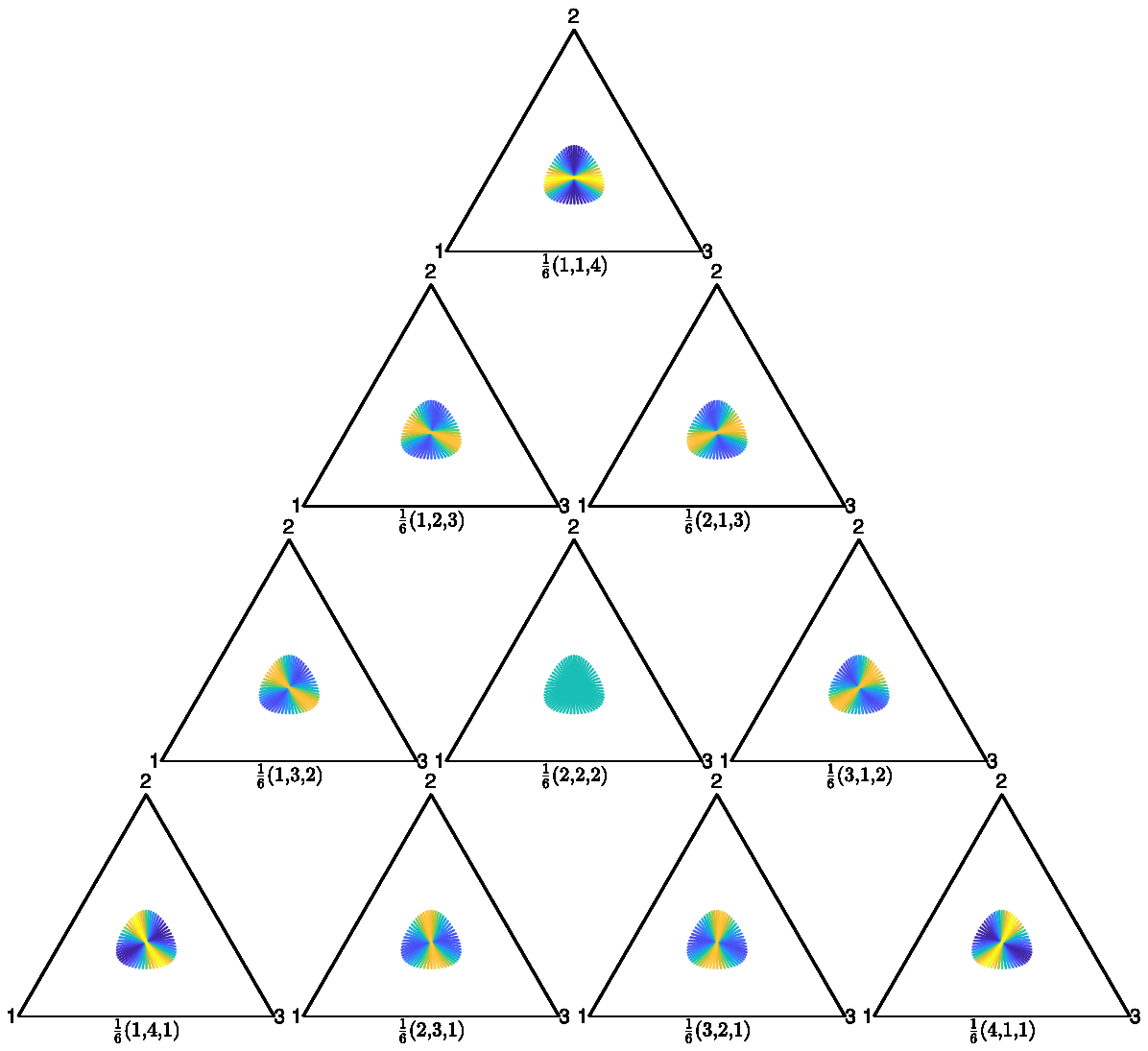}};
	\node[] at (6,0){\includegraphics[clip=true, trim=0cm 5.5cm 0cm 5.5cm,width=6cm]{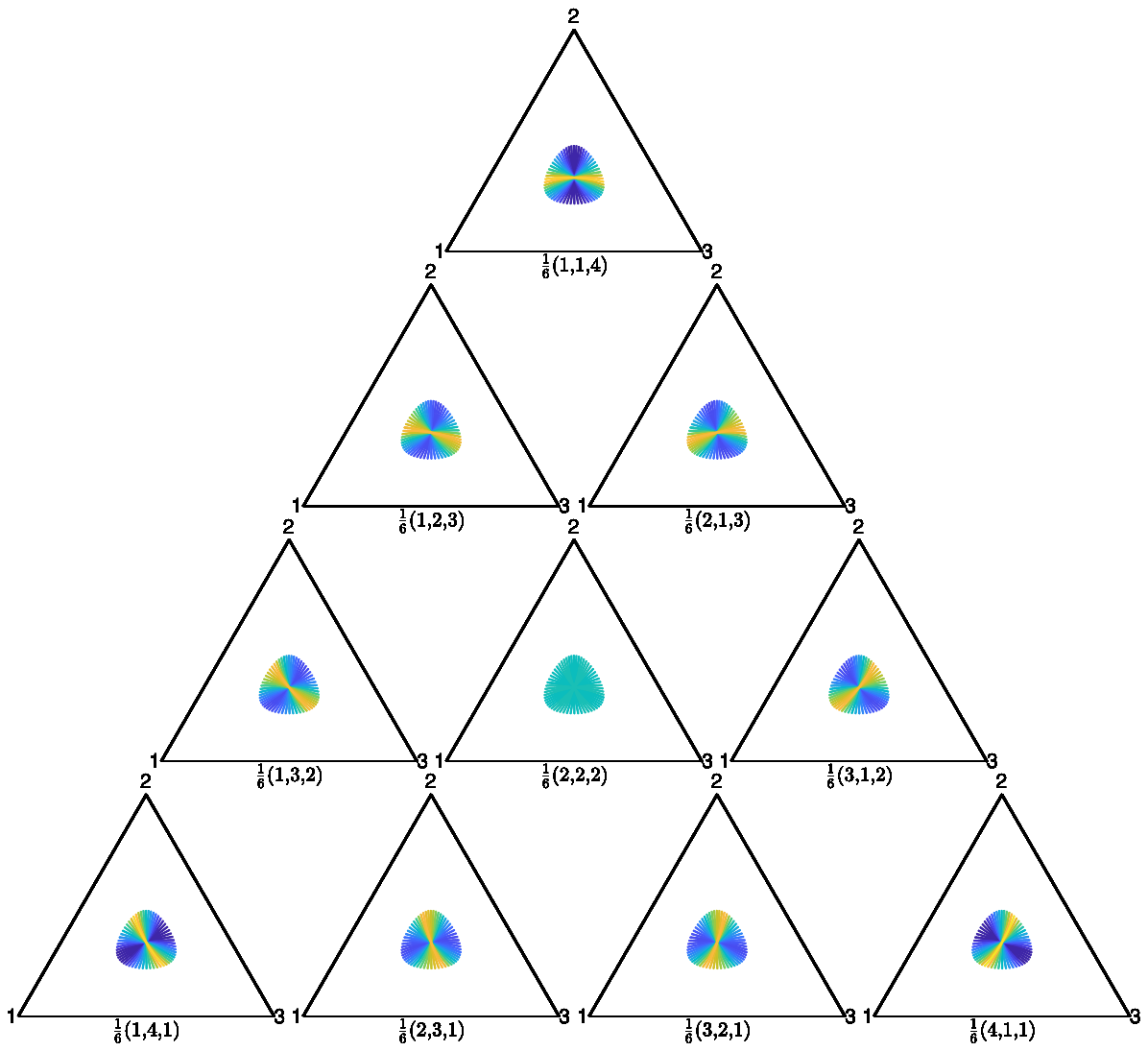}};
	\node[] at (3,2.5){\includegraphics[clip=true, trim=2.5cm 10.5cm 2cm 10cm,width=4cm]{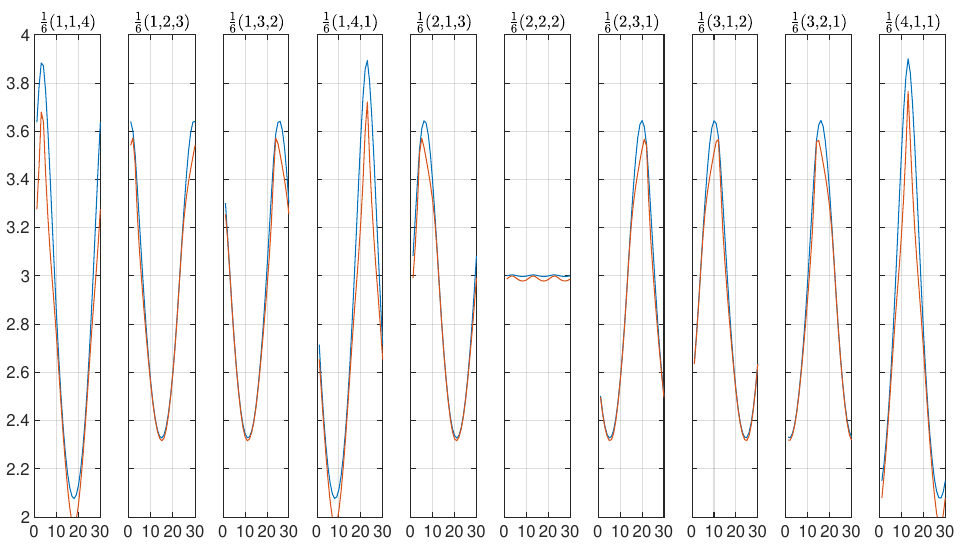}};
	\end{tikzpicture}
}
	\caption{Similar to Figure~\ref{fig:experimen1} but with $\Theta = [-1/2,1/2]$. 
	Note how on this tight parameter domain around $\theta=0$ (the value of the reference measure), the Ricci curvature lower bound gives a very close lower bound on the minimum rate of convergence for each of the models. 
	The middle shows the direct comparison of the two values across the 30 exponential families. The minimum rate of convergence is shown in blue, and the Hessian in red. 
	}
	\label{fig:experiment1s}
\end{figure}

\begin{figure}[h]
\scalebox{1.2}{ 
	\begin{tikzpicture}[x=1cm,y=1cm]
	\node[] at (0,0){\includegraphics[clip=true, trim=0cm 5.5cm 0cm 5.5cm,width=6cm]{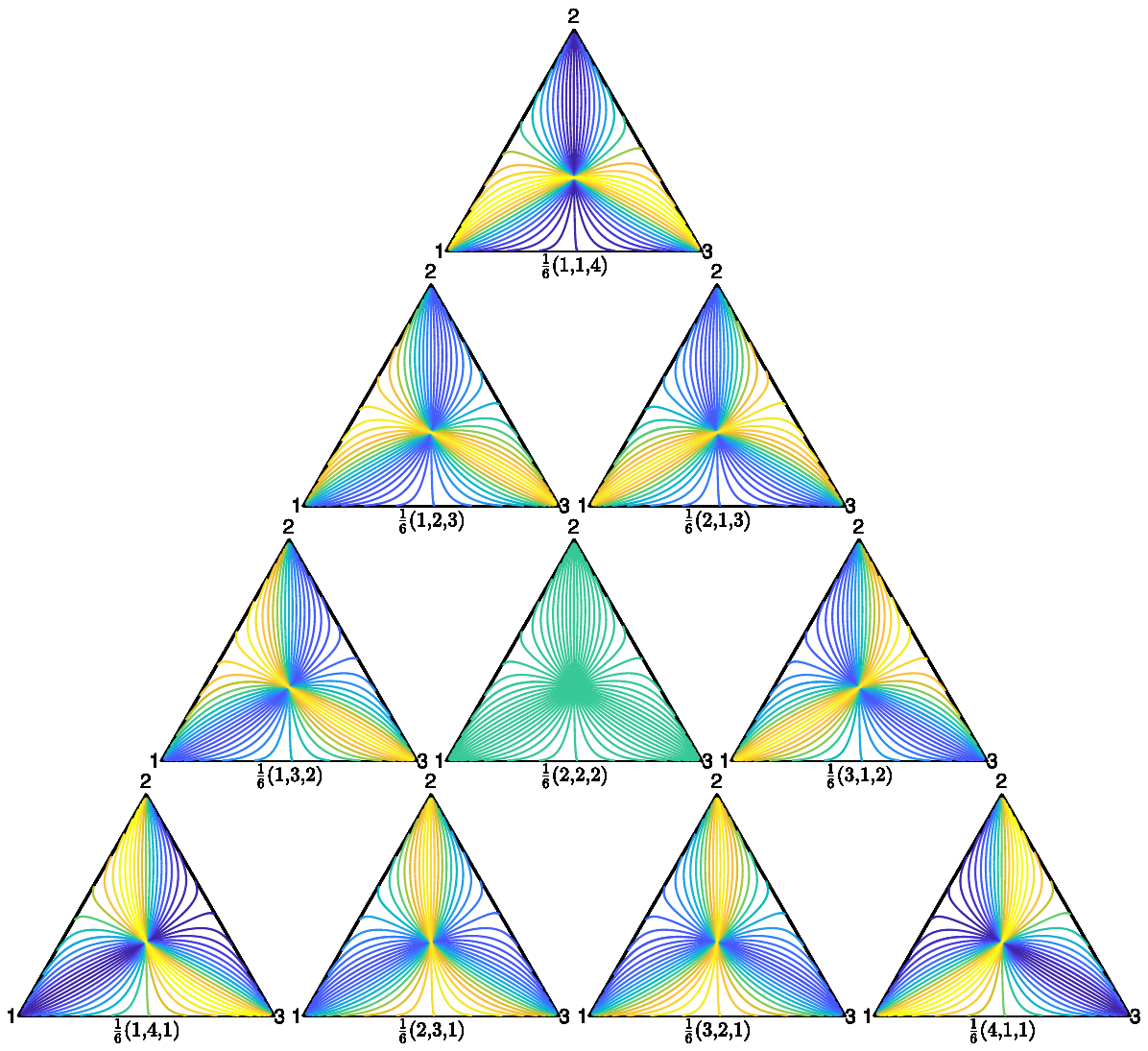}};
	\node[] at (6,0){\includegraphics[clip=true, trim=0cm 5.5cm 0cm 5.5cm,width=6cm]{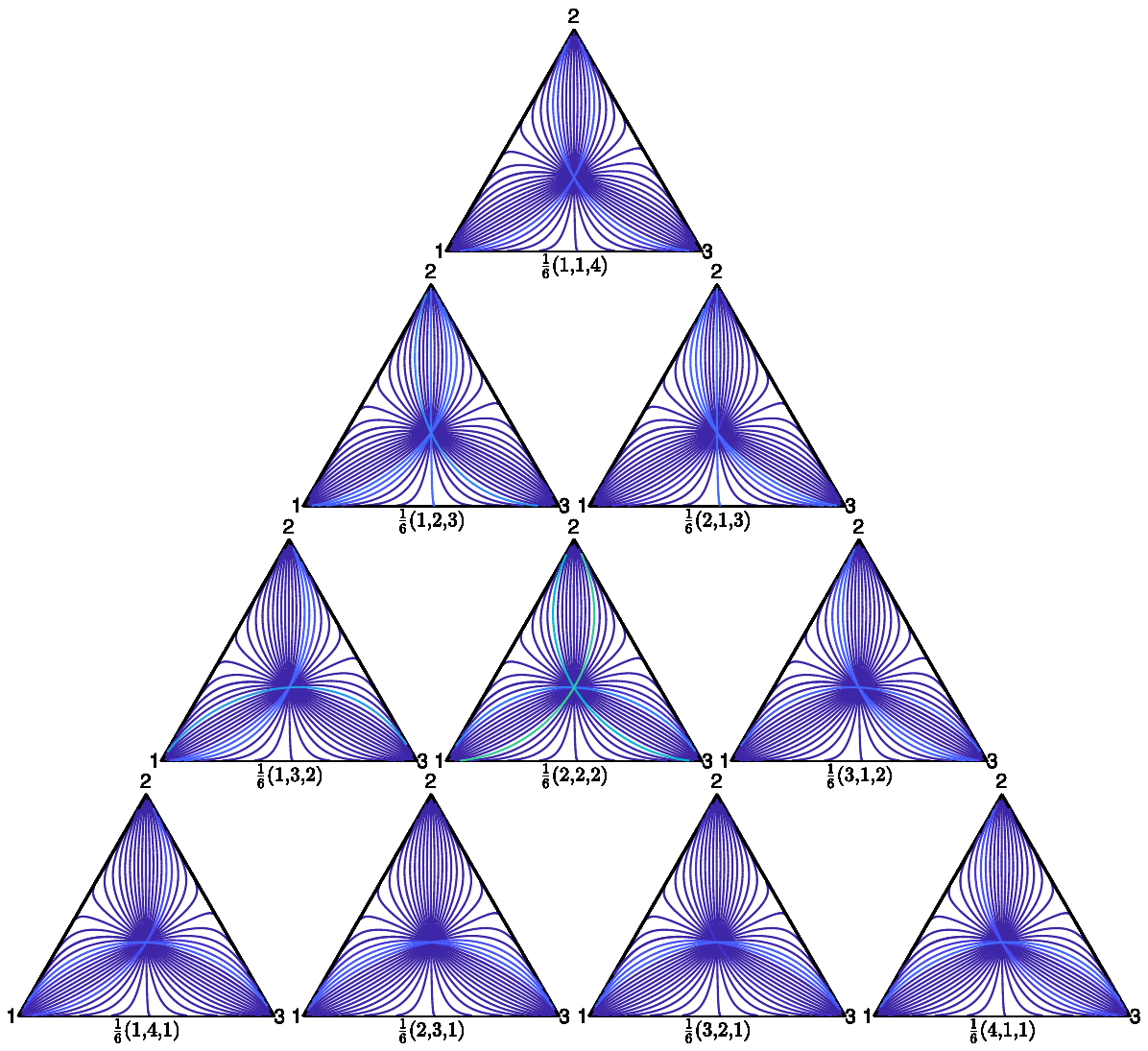}};
	\node[] at (3,2.5){\includegraphics[clip=true, trim=2.5cm 10.5cm 2cm 10cm,width=4cm]{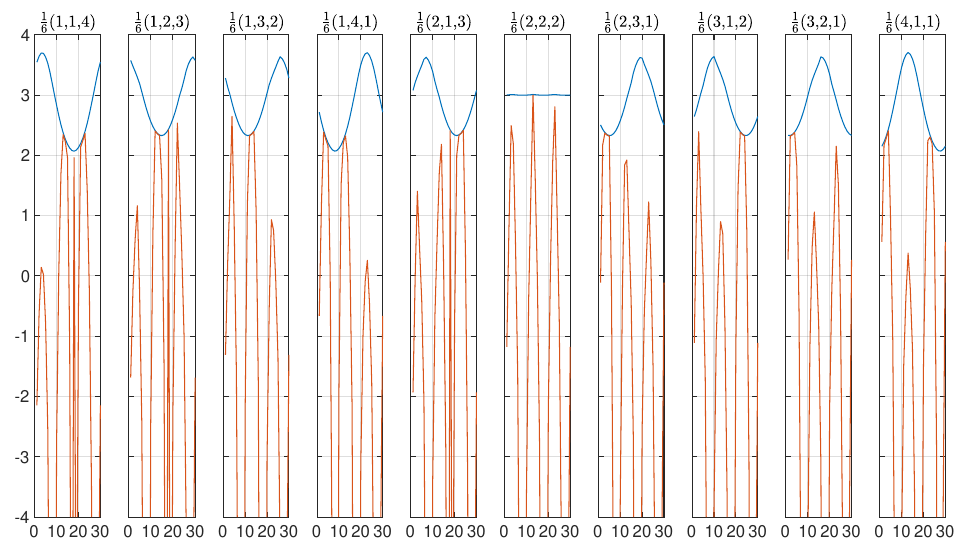}};
	\end{tikzpicture}
}	
	\caption{Similar to Figures~\ref{fig:experimen1}, but with a larger parameter domain $\Theta = [-4,4]$. 
		On this relatively large parameter domain, the models contain points close to the boundary of the simplex, where the Hessian (and the Ricci curvature) can have large oscillations. 
		In turn, we observe larger gaps to the minimum rate of convergence, compared with Figure~\ref{fig:experiment1s}. 
	}
	\label{fig:experiment1l}
\end{figure}

\begin{figure}
\centering
\includegraphics[clip=true,trim=2cm 8cm 2cm 8cm,width=.8\linewidth]{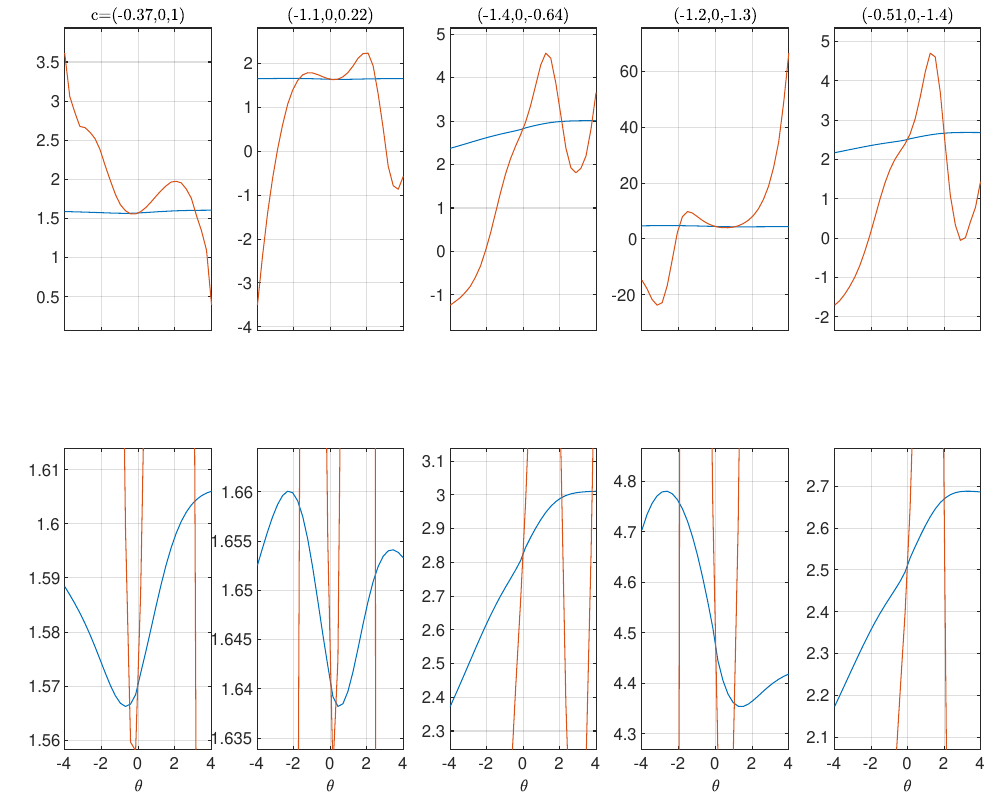}
\caption{Convergence rates and minimum Hessian eigenvalue at individual parameter choices. 
Here we fixed the ground metric  $\omega=(\omega_{12},\omega_{23},\omega_{13})=(1/2,1/2,0)$. 
Each subplot corresponds to one exponential family, with sufficient statistic indicated at the top. 
Within a region around $\theta=0$ (the value of the reference measure), the minimum of the Hessian is closer to the convergence rates. 
In fact, the Hessian eigenvalue intersects the rate of convergence at $\theta=0$. The Hessian at $\theta=0$ is the asymptotic rate of convergence. 
The lower row zooms in the y axis of the upper row. 
For these exponential families, the convergence rates do not vary much across choices of the initial parameter value. 
}
\label{fig:experiment1pointwise}
\end{figure}

\end{document}